\documentclass[12pt]{amsart}
\usepackage{lmodern}
\usepackage{amssymb,amsmath,amsthm}
\usepackage{enumerate}
\usepackage{enumitem}

\usepackage[margin=3cm]{geometry} 
\usepackage{xcolor}
\usepackage{multicol}
\usepackage[colorlinks = true, linkcolor = blue, citecolor=red]{hyperref}
\usepackage{comment}
\usepackage[capitalise]{cleveref}
\usepackage{bm}
\usepackage{bbm}

\theoremstyle{plain}
\newtheorem{theorem}{Theorem}[section]

\newtheorem{lemma}[theorem]{Lemma}
\newtheorem{corollary}[theorem]{Corollary}

\newtheorem{proposition}[theorem]{Proposition}

\theoremstyle{definition}
\newtheorem{definition}[theorem]{Definition}

\newtheorem{problem}{Problem}[section]
\newtheorem*{remark}{Remark}
\newtheorem*{remarks}{Remarks}

\numberwithin{equation}{section}

\newcommand{\im}[1]{\mathbf{ i_{#1}}}
\newcommand{\ims}[1]{\mathbf{ i_{#1}^2}}
\newcommand{\jm}[1]{\mathbf{ j_{#1}}}
\newcommand{\jms}[1]{\mathbf{ j_{#1}^2}}
\newcommand{\km}[1]{\mathbf{k_{#1}}}
\newcommand{\kms}[1]{\mathbf{k_{#1}^2}}

\newcommand{\be}[1]{\mathbf{e_{#1}}}
\newcommand{\bec}[1]{\mathbf{\overline{e}_{#1}}}
\newcommand{\bep}[1]{\bm{\varepsilon_{#1}}}

\newcommand{\mB}{\mathbb{B}}
\newcommand{\mC}{\mathbb{C}}

\newcommand{\mR}{\mathbb{R}}

\newcommand{\mM}{\mathbb{M}}
\newcommand{\mI}{\mathbb{I}}

\newcommand{\cL}{\mathcal{L}}
\newcommand{\cB}{\mathcal{B}}

\newcommand{\ra}{\rightarrow}
\newcommand{\Ra}{\Rightarrow}

\DeclareMathOperator{\spn}{span}

\DeclareMathOperator{\hol}{Hol}

\begin{document}

\title[Cauchy--Riemann Equations and Involutions]{Classes of Holomorphic Multicomplex-Valued Functions Generated by Elliptic-Admissible Involutions}

\author{Nicolas Doyon}
\address{Nicolas Doyon: Département de mathématiques et de statistique, Université Laval, Quebec City, QC, G1V 0A6, Canada}
\email{Nicolas.Doyon@mat.ulaval.ca}

\author{Pierre-Olivier Paris\'e}
\address{Pierre-Olivier Parisé: Département de mathématiques et d'informatique, Université du Québec à Trois-Rivières,
Trois-Rivières, QC, G8Z 4M3, Canada}
\email{pierre-olivier.parise@uqtr.ca}
\thanks{POP is supported by an NSERC discovery grant with {\#}RGPIN-2026-04740 and {\#}DGECR-2026-00390.}

\author{William Verreault}
\address{
William Verreault: Department of Mathematics, University of Toronto, Toronto, ON, M5S 2E4, Canada}
\email{william.verreault@utoronto.ca}

\begin{abstract}
We classify and count the real-algebra involutions of the multicomplex
algebra $\mM\mC(n)$ that map each element of its canonical monomial
basis to a signed monomial, using a matrix model over $\mathbb F_2$.

An \emph{elliptic-admissible pair} consists of such an involution
$\sigma$ and a monomial unit $\im{}\in\mI(n)$ satisfying
$\ims{}=-1$ and $\sigma(\im{})=-\im{}$. The fixed algebra of $\sigma$
is a real form for the complex structure $\cL_{\im{}}$ defined by
multiplication by $\im{}$, and the pair yields a Cauchy--Riemann system.
We prove that its solution class depends only on $\im{}$, not on
$\sigma$, and is exactly the class of mappings holomorphic with respect
to $\cL_{\im{}}$. Multicomplex holomorphy is recovered as the
intersection of the classes associated with the elementary generators.
We obtain the analogous characterization for anti-holomorphic classes,
describe twisted systems intertwining two such complex structures, and
prove that every $\im{}$-holomorphic or $\im{}$-anti-holomorphic
mapping is componentwise harmonic (solutions to Laplace's equation).
\end{abstract}



\maketitle


\section{Introduction}\label{S:Introduction}

In the classical theory of complex-valued functions, a function $f : \Omega \subseteq \mC \ra \mC$ that is holomorphic can be characterized through a system of equations called the Cauchy--Riemann equations. These equations can be described using the Wirtinger \cite{Remmert1991}  operators 
    \[
        \partial_z := \frac{1}{2} \Big( \frac{\partial}{\partial x} - i \frac{\partial}{\partial y} \Big) \quad \text{and} \quad \partial_{\overline{z}} := \frac{1}{2} \Big( \frac{\partial}{\partial x} + i \frac{\partial}{\partial y} \Big),
    \]
where $z := x + iy$, with $x, y \in \mR$ and $i^2 = -1$, and $\overline{z} := x - iy$ denotes the complex conjugate of $z$. A continuously differentiable function $f : \Omega \subseteq \mC \ra \mC$ is then holomorphic if and only if $\partial_{\overline{z}} f = 0$.

Several approaches generalize holomorphic-function theory through systems of partial
differential equations. One such approach uses Clifford and complexified Clifford
algebras. In this setting, the analogue of $\partial_{\overline{z}}$ is a Dirac
operator: a linear combination of the partial derivatives associated with the
anticommuting generators of the Clifford algebra, whose null solutions are called
monogenic functions. The idea of a Dirac-type system for space goes back to Moisil and
Teodorescu \cite{MoisilTeodorescu1931}, and the function theory of a single quaternionic
variable to Fueter \cite{Fueter1935, Fueter1936}, whose theory of regular functions was
placed on its modern footing by Sudbery \cite{Sudbery1979}. It was then systematically
developed in the setting of Clifford algebras by Delanghe \cite{Delanghe1970}, by
Brackx, Delanghe and Sommen \cite{BrackxDelangheSommen1982} and by Delanghe, Sommen and
Sou\v{c}ek \cite{DelangheSommenSoucek1992} (see also Gilbert and Murray
\cite{GilbertMurray1991}), and extended to the complexified Clifford algebras by Ryan
\cite{Ryan1982}. A more recent and rapidly developing branch replaces the Dirac system
by a slice condition, giving the slice-regular functions of a quaternionic variable
introduced by Gentili and Struppa \cite{GentiliStruppa2007} and the slice-monogenic
functions over Clifford algebras of Colombo, Sabadini and Struppa
\cite{ColomboSabadiniStruppa2009, ColomboSabadiniStruppa2011}.

Unlike the Clifford approach, in which a single first-order operator governs a non-commutative, multiplicatively non-closed class, a second approach uses commutative generators to define a higher-dimensional commutative real algebra and to create a class of functions as solutions to a system of Cauchy--Riemann equations.

This second direction was pursued for the bicomplex numbers, that is, the set $\mM \mC(2)$ (also denoted by $\mB \mC$) of expressions of the form $\eta = x + y \im{1} + z \im{2} + w \im{1}\im{2}$, where $x, y, z, w \in \mR$, and where the imaginary units $\im{1}$ and $\im{2}$ commute and satisfy $\ims{1} = \ims{2} = -1$. 
In \cite{Bicomplex}, the authors introduced three conjugations of the set of bicomplex numbers (denoted by $\dagger$, $\ast$ and $\overline{\phantom{2}}$) and defined the four associated Wirtinger-type operators\footnote{The precise expression of these operators can be found originally in \cite{RochonShapiro}.} $\partial_{\eta}$, $\partial_{\eta^\dagger}$, $\partial_{\eta^\ast}$, and $\partial_{\overline{\eta}}$, so that a bicomplex-valued function $f : \Omega \subseteq \mM \mC(2) \ra \mM \mC(2)$ that is continuously differentiable as a map of real vector spaces is holomorphic if and only if the following Cauchy--Riemann type equations $\partial_{\eta^\dagger} f = \partial_{\eta^\ast} f = \partial_{\overline{\eta}} f = 0$ are satisfied. These Wirtinger operators were then generalized to the multicomplex numbers $\mM \mC(n)$ \cite{Baley}, a generalization of the complex numbers and the bicomplex numbers to higher dimensions, using a distinguished family of conjugations \cite{Struppa2012}.

The unifying algebraic notion behind all the conjugations mentioned above is that of an involution. Terminology varies in the literature. Throughout this paper, an involution refers to a real-algebra automorphism whose square is the identity. This convention is consistent with the usage in \cite{SangwineTodd2007,Lawson2021,Parise}. Our first result in Section~\ref{S:InvoPreserveElemUnits} shows that there are more involutions than the $2^n$ mainly referred to in the literature on multicomplex function theory; see \cite{ CourchesneTremblay2025, Pelletier2009, Bicomplex, RochonShapiro, RochonTremblay2004, Struppa2012}. This suggests that there are more Cauchy--Riemann type equations that can be created from these additional involutions. Based on these new involutions, we introduce a new framework—a complex structure on $\mM \mC (n)$ associated with an involution and a distinguished unit $\im{} \in \mM \mC (n)$ with $\ims{} = -1$—to study these systems of Cauchy--Riemann equations and to establish a connection to the space of holomorphic functions of a multicomplex variable. This approach is presented in Section \ref{S:classesMonoFunctions}, and to the best of the authors' knowledge, the approach presented here has not been developed in the literature. We hope it will shed new light on function theory in multicomplex space. 

For the rest of the introduction, we let $\im{1}, \ldots, \im{n}$ be the elementary commuting imaginary units of $\mM \mC(n)$ and define the set $\mI(n)$ as the set of numbers that can be written as $\im{1}^{a_1}\cdots \im{n}^{a_n}$ with $a_k\in \{0,1\}$. We also denote by $\pm \mI(n)$ the set of all the numbers of the form $\pm \im{}$, for $\im{} \in \mI(n)$. Section~\ref{S:Multicomplex} gives some preliminaries on the multicomplex numbers. 

\subsection{Summary of the main results}
In Section~\ref{S:Multicomplex}, we introduce the concept of holomorphic functions for multicomplex-valued functions of a multicomplex variable in terms of the multicomplex linearity of the real differential. In Proposition \ref{P:DerivativeEquivalence}, we show that this definition of holomorphy is equivalent to the one given in the literature in terms of the limit of a difference quotient. This fact is well-known in the literature for bicomplex-valued functions of a bicomplex variable (see \cite[Remark 7.3.4]{Bicomplex}). We provide a complete proof in the multicomplex setting, without assuming any sort of regularity. 

In Section~\ref{S:InvoPreserveElemUnits}, we introduce the notion of an $\mI(n)$-preserving involution of $\mM \mC(n)$, $n \geq 1$, that is, an involution $f$ of $\mM \mC(n)$ that maps every element of $\mI(n)$ into $\pm\mI(n)$. We show that the number of $\mI(n)$-preserving involutions of the multicomplex numbers is given by the formula
    \begin{align}
        \sum_{k=\lceil n/2\rceil}^{n}  \Big( \prod_{j = 1}^{k-1} \frac{2^n - 2^j}{2^k - 2^j} \Big) \Big( \prod_{j = 0}^{n-k-1} (2^k - 2^j) \Big) 2^k , \label{Eq:NumberOfInPreservingInvo}
    \end{align}
where an empty product is understood to be equal to $1$. This gives an integer sequence starting with $2$, $6$, $44$, $576$, $15392$, $\ldots$, such that, surprisingly, we could not find a reference for it in the On-Line Encyclopedia of Integer Sequences (OEIS). 
We prove formula \eqref{Eq:NumberOfInPreservingInvo} by translating the counting problem into a counting problem in the world of matrix theory over the field of two elements $\mathbb{F}_2$ and using linear algebra tools there. 
The proof reveals that the inner factor $\prod_{j=1}^{k-1}(2^n-2^j)/(2^k-2^j)$ is the Gaussian binomial coefficient $\genfrac{[}{]}{0pt}{}{n-1}{k-1}_2$, so that \eqref{Eq:NumberOfInPreservingInvo} is a weighted sum over the subspace lattice of $\mathbb{F}_2^{n}$. 
Such sums are ubiquitous in the applied theory of linear algebra over $\mathbb{F}_2$. 
On the one hand, the subspace lattice of $\mathbb{F}_2^n$ is the ambient space of the subspace codes used in random network coding \cite{KoetterKschischang2008}. On the other hand, involutory matrices over fields of characteristic two are of independent interest in symmetric-key cryptography, where they allow a block cipher to reuse the same circuitry for encryption and decryption \cite{GuptaPandeyRaySamanta2019}. 

Furthermore, one might be tempted to consider all involutions of the space of multicomplex numbers instead, but some of the resulting involutions lose their connection with, and their significance for, the Cauchy--Riemann equations. 
To remain consistent with the defining features of the subset of involutions used in \cite{Struppa2012}, we therefore directed our efforts toward the characterization of the involutions preserving the elementary units of the multicomplex numbers. 
Interestingly, this turns out to be a more difficult problem than the general one of counting all involutions; the latter is briefly discussed in Section~\ref{Ss:GeneralCase} for completeness.

In Section~\ref{S:classesMonoFunctions}, we define the concepts of \emph{elliptic-admissible} involutions ($e$-admissible involutions, for short)  and of \emph{$e$-admissible} pair $(\sigma, \im{})$. We say that $\sigma$ is $e$-admissible if there exists $\im{} \in \mI (n)$ such that $\sigma (\im{}) = - \im{}$ and that the pair $(\sigma, \im{})$ is an \emph{$e$-admissible pair} if $\sigma$ is an $e$-admissible involution, and the unit $\im{}$ satisfies $\ims{} = -1$ and $\sigma (\im{}) = - \im{}$ (see Definition \ref{D:EAdmissible}). 
We then attach a first-order system of Cauchy--Riemann type equations to every $e$-admissible pair $(\sigma, \im{})$, and identify precisely the class of functions it defines. Our Theorem \ref{T:SigmaIndependence} shows a striking phenomenon : On the one hand, the involution $\sigma$ serves to define a coordinate system of the complex structure to express the new Cauchy--Riemann equations; on the other hand, the function theory is solely described by the imaginary unit $\im{}$ and therefore the class of functions described by the Cauchy--Riemann equations is independent of the choice of the involution $\sigma$. This phenomenon justifies the introduction of the class of $\im{}$-holomorphic functions, with $\ims{} = -1$ and $\im{} \in \mI (n)$, and our Theorem \ref{T:EquivalenceSpanSWithHolomorphy} gives a sharp characterization of the class of holomorphic functions of a multicomplex variable in terms of the classes of $\im{}$-holomorphic functions. We also illustrate the theory completely for bicomplex-valued functions. 
We then describe how the framework developed in this paper can be used to describe families of $\im{}$-anti-holomorphic functions. Our Proposition \ref{P:AntiHolo} then shows how a class of anti-holomorphic functions can be defined, recovering the familiar fact that a function of one complex variable that is both holomorphic and anti-holomorphic is constant. We also introduce twisted classes obtained by allowing the differential to intertwine two different complex structures. Proposition \ref{P:Twisted} shows that every such class is an automorphic image of an $\im{}$-holomorphy class.

Finally, we show that the Wirtinger operators defined from the complex structure can be combined to obtain the Laplacian, justifying the terminology used for elliptic-admissible involutions. We then show in Corollary \ref{C:HolHarmonic} that any $\im{}$-holomorphic or $\im{}$-anti-holomorphic function is harmonic (solutions to Laplace's equation). 

\section{Background on multicomplex numbers} \label{S:Multicomplex}
In 1892, Segre \cite{Segre} introduced an algebraic structure that he called $n$-complex numbers with the goal of defining a multiplication operation between vectors of $\mC^n$, for $n \geq 2$. 

Interest in the theory of $n$-complex numbers (nowadays referred to as multicomplex numbers) and its applications have grown over the past decades. For example, they are used to introduce generalizations of concepts from real and complex analysis, e.g., multicomplex fractional operators \cite{Ceballos2022}, multicomplex hyperanalytic functions \cite{vajiac}, Laurent series \cite{Luna2017}, Riemannian and semi-Riemannian geometry \cite{vajiac2018}, and multicomplex holomorphic functions \cite{Struppa2012}. We also mention the use of multicomplex numbers to generalize the Mandelbrot set to higher dimensions \cite{Parise2019, BrouilletteRochon2019, Pelletier2009, Rochon1}, in theoretical physics to generalize the linear and nonlinear {S}chr\"{o}dinger equations \cite{RochonTremblay2004, Theaker2017}, and in machine learning to generalize complex-valued neural networks  \cite{AlpayDikiVajiac2023}.

A modern treatment of these numbers is presented in \cite{Baley} with a preface describing the history of the development of associative algebras.  We will mainly follow the presentation given in \cite{BrouilletteRochon2019}, with slight changes in notation.

\subsection{Multicomplex numbers}
The definition of the multicomplex numbers is given recursively. Let $\mM \mC(0)$ be the set of real numbers and let $\mM \mC(n)$, $n \geq 1$, be the set
    \begin{align}
        \mM \mC(n) := \{ \eta = \eta_1 + \eta_2 \im{n}\, :\, \eta_1 , \eta_2 \in \mM \mC(n- 1 ) ,\, \ims{n} = -1 \},
    \end{align}
where $\im{n}$ is a formal square root of $-1$. For example, when $n = 1$, we obtain the set $\mM \mC(1)$ of complex numbers $\eta_1 + \eta_2 \im{1}$, where $\ims{1} = -1$. When $n = 2$, we obtain the set $\mM \mC(2)$ of bicomplex numbers $\eta_1 + \eta_2 \im{2}$, where $\eta_1 , \eta_2$ are complex numbers, and $\im{2}$ is a new square root of $-1$ such that $\ims{2} = -1$, and $\im{1} \neq \im{2}$. 

We say that two multicomplex numbers $\eta$ and $\zeta$ are equal if and only if $\eta_1 = \zeta_1$ and $\eta_2 = \zeta_2$. If we let $\eta_2 = 0$ in the expression of a multicomplex number $\eta = \eta_1 + \eta_2 \im{n}$, we see that $\mM \mC(n - 1) \subseteq \mM \mC(n )$. The set of multicomplex numbers becomes a commutative ring if we endow it with the following algebraic operations:
    \begin{enumerate}[label=\arabic*)]
    \item $\eta + \zeta := (\eta_1 + \zeta_1 ) + (\eta_2 + \zeta_2) \im{n}$;
    \item $\eta \zeta := (\eta_1 \zeta_1 - \eta_2 \zeta_2) + (\eta_1 \zeta_2 + \eta_2 \zeta_1) \im{n}$.
    \end{enumerate}
    These last operations must be understood recursively.
    
    Let $\eta = \eta_1 + \eta_2 \im{n}$ be a multicomplex number. 
  Then $\eta_1 , \eta_2 \in \mM \mC(n - 1)$, so there are multicomplex numbers $\eta_{11}, \eta_{12}, \eta_{21}, \eta_{22}  \in \mM \mC(n - 2 )$ 
 such that $\eta_1 = \eta_{11} + \eta_{12} \im{n-1}$ and $\eta_2 = \eta_{21} + \eta_{22} \im{n-1}$. Replacing the $\eta_1$ and $\eta_2$ in the expression for $\eta$, we obtain the representation of $\eta\in\mM \mC(n)$ in terms of four components in $\mM \mC(n - 2)$,
        \begin{align*}
        \eta=(\eta_{11} + \eta_{12} \im{n - 1} ) + (\eta_{21} + \eta_{22} \im{n - 1}) \im{n}.
        \end{align*}
    From the definition of the multiplication, we can distribute $\im{n}$ to obtain
        \begin{align*}
        \eta=\eta_{11} + \eta_{12} \im{n-1} + \eta_{21} \im{n} + \eta_{22} \im{n-1} \im{n}.
        \end{align*}
    For example, a bicomplex number $\eta = \eta_1 + \eta_2 \im{2}$ can be expressed as a linear combination involving four real coefficients,
        \begin{align*}
        \eta = \eta_{11} + \eta_{12} \im{1} + \eta_{21} \im{2} + \eta_{22} \im{1}\im{2}.
        \end{align*}
    We can continue this process recursively until we reach the set $\mM \mC(0)$. At each stage $k$ ($1 \leq k \leq n$) of the process, we obtain a representation of a multicomplex number in terms of $2^{k}$ multicomplex numbers in $\mM \mC(n - k)$.  
    
    The representation we are interested in is the one in terms of the $2^n$ components in $\mM \mC(0)$. To be more explicit, recall that $\mI(n)$ is the set of all different possible products of the elements in the set $\{ 1 , \im{1} , \im{2} , \ldots , \im{n} \}$. Since multiplication is commutative, the cardinality of $\mI(n)$ is $2^{n}$. The elements in this set can be used as a real basis, the \emph{monomial real basis}, to express a multicomplex number in the following way: 
    \begin{align}
        \eta = \sum_{\im{} \in \mI(n)} \eta_{\im{}} \im{}, \label{Eq:CanoFormMC}
        \end{align}
    where $\eta_{\im{}} \in \mR$. This is called the \emph{canonical real representation} or \emph{the real monomial basis representation}. Using this representation and the algebraic operations defined above, we can view the set $\mM \mC(n )$ as a commutative real algebra. 

    There is an efficient way to formulate the canonical representation of a multicomplex number using the vector space $\mathbb{F}_2^n$, where $\mathbb F_2$ is the field with two elements $\{ 0, 1 \}$. This will be used later in the subsection on multicomplex holomorphy and throughout Section \ref{S:InvoPreserveElemUnits} where the $\mI (n)$-preserving involutions are characterized. 
    
    For $\vec b=(b_1,\ldots,b_n)^\top\in\mathbb F_2^n$, we introduce the following notation which emphasizes which generators from the set $\{ \im{1} , \im{2} , \ldots , \im{n} \}$ are selected to form the elements of the set $\mI (n)$: 
\[
\mathbf i^{\vec b}:=\im{1}^{b_1}\im{2}^{b_2}\cdots \im{n}^{b_n}.
\]
Multiplying two such monomials and using the relations $\ims{k}=-1$, $1 \leq k \leq n$, we obtain the multiplication rule
\begin{equation}\label{Eq:MonomialMult}
\mathbf i^{\vec a}\,\mathbf i^{\vec b}=(-1)^{\vec a^{\top} \vec b}\,\mathbf i^{\vec a\oplus\vec b}
\qquad (\vec a,\vec b\in\mathbb F_2^n),
\end{equation}
where $\vec a\oplus\vec b$ denotes the sum in $\mathbb F_2^n$ and $\vec a^{\top}\vec b=\sum_{k=1}^n a_k b_k$ is computed modulo $2$ (the quantity $\vec{a}^\top \vec{b}$ represents the number of components in $\vec{a}$ and $\vec{b}$ that are simultaneously $1$ and therefore contribute to a sign change in the product). The canonical real representation of a multicomplex number $\eta$ can be written as follows: 
\begin{equation}\label{Eq:CartesianMonomialF2n}
    \eta = \sum_{\vec b \in \mathbb F_2^n} x_{\vec b}\, \mathbf i^{\vec b},
\end{equation}
where $x_{\vec b} := \eta_{\mathbf i^{\vec b}} \in \mR$. 

\subsection{An idempotent representation for multicomplex numbers}
Of particular importance in the set of multicomplex numbers are the numbers $\eta$ such that $\eta^2=\eta$, which are called idempotent numbers. We assume for now that $n \geq 2$ and we consider
		\begin{align*}
		\be{n} := \frac{1 + \im{n-1} \im{n}}{2} \quad \text{and} \quad \bec{n} := \frac{1 - \im{n-1} \im{n}}{2} .
		\end{align*}
	An additional property that these numbers have is that $\be{n} \bec{n} = 0$. If we multiply a multicomplex number $\eta = \eta_1 + \eta_2 \im{n}$ by $\be{n}$ and by $\bec{n}$ respectively, we obtain
		\begin{align*}
		\eta \be{n} = (\eta_1 - \eta_2 \im{n-1}) \be{n} \quad \text{and} \quad \eta \bec{n} = (\eta_1 + \eta_2 \im{n - 1} ) \bec{n}.
		\end{align*}
	Since $\be{n} + \bec{n} = 1$, summing $\eta\be{n}$ and $\eta\bec{n}$ yields the idempotent representation of a multicomplex number, namely
		\begin{align*}
		\eta = (\eta_1 - \eta_2 \im{n - 1}) \be{n} + (\eta_1 + \eta_2 \im{n - 1} ) \bec{n} .
		\end{align*}
	We see that the numbers multiplying $\be{n}$ and $\bec{n}$ are elements of $\mM \mC (n - 1 )$, which we call the idempotent components of $\eta$. We will denote them by $\eta_{\be{n}}$ and $\eta_{\bec{n}}$, respectively. The idempotent representation can therefore be rewritten as
		\begin{align}
		\eta = \eta_{\be{n}} \be{n} + \eta_{\bec{n}} \bec{n} . \label{Eq:IdempotentRepresentation}
		\end{align}
  Note that two multicomplex numbers are equal if and only if their idempotent components are equal.
  
		The idempotent representation is important because it transforms the multiplication of multicomplex numbers into a componentwise multiplication. More precisely, if $\eta = \eta_{\be{n}} \be{n} + \eta_{\bec{n}} \bec{n}$ and $\zeta = \zeta_{\be{n}} \be{n} + \zeta_{\bec{n}} \bec{n}$, then we have
		\begin{align}
		\eta \zeta = \eta_{\be{n}} \zeta_{\be{n}} \be{n} + \eta_{\bec{n}} \zeta_{\bec{n}} \bec{n} . \label{Eq:IdempotentMultiplication}
		\end{align}
	
	We now apply the idempotent representation of elements of the set $\mM \mC (n - 1)$ to the idempotent components of a multicomplex number $\eta \in \mM \mC (n)$. Define
		\begin{align*}
		\be{n-1} := \frac{1 + \im{n-2}\im{n-1}}{2} \quad \text{and}\quad  \bec{n-1} := \frac{1 - \im{n-2}\im{n-1}}{2} .
		\end{align*}
	Then, the idempotent components $\eta_{\be{n}}$ and $\eta_{\bec{n}}$ of $\eta \in \mM \mC (n)$ can be written as
		\begin{align*}
		\eta_{\be{n}} = \eta_{\be{n-1}\be{n}} \be{n-1} + \eta_{\bec{n-1}\be{n}} \bec{n-1}
		\end{align*}
	and
		\begin{align*}
		\eta_{\bec{n}} = \eta_{\be{n-1}\bec{n}} \be{n-1} + \eta_{\bec{n-1}\bec{n}} \bec{n-1},
		\end{align*}
	where $\eta_{\be{n-1}\be{n}}, \eta_{\bec{n-1} \be{n}}, \eta_{\be{n-1} \bec{n}} , \eta_{\bec{n-1} \bec{n}} \in \mM \mC (n - 2)$. Replacing these in the idempotent representation of $\eta \in \mM \mC (n)$, we obtain a second idempotent representation in terms of components in $\mM \mC (n - 2 )$,
		\begin{align*}
		\eta = \eta_{\be{n-1}\be{n}} \be{n-1} \be{n} + \eta_{\bec{n-1} \be{n}} \bec{n-1} \be{n} + \eta_{\be{n-1} \bec{n}} \be{n-1} \bec{n} + \eta_{\bec{n-1}\bec{n}} \bec{n-1} \bec{n} .
		\end{align*}
	More generally, define the following elements for each integer $ k \geq 2$:
		\begin{align*}
		\be{k} := \frac{1 + \im{k-1}\im{k}}{2} \quad \text{and} \quad \bec{k} := \frac{1 - \im{k-1}\im{k}}{2} .
		\end{align*}
	We then define a family of sets $\mathcal{E} (k, n)$ inductively for $n \geq 2$ and $2 \leq k \leq n$:
		\begin{enumerate}[label=\arabic*)]
		\item $\mathcal{E} (n, n) := \{ \be{n} , \bec{n} \}$ for $k = n$;
		\item $\mathcal{E} (k, n) := \mathcal{E} (k + 1, n) \be{k} \cup \mathcal{E} (k + 1 , n) \bec{k}$ for $2 \leq k < n$.
		\end{enumerate}
	Now, for any $2 \leq k \leq n$, an induction argument shows that the cardinality of $\mathcal{E} (k, n )$ is $2^{n - k + 1}$. Also, by induction, we have that
		if $\bep{} \in \mathcal{E}(k, n)$, then $\bep{}^2 = \bep{}$, and
		if ${\bep{}}_1 , {\bep{}}_2 \in \mathcal{E} (k, n)$ with ${\bep{}}_1 \neq {\bep{}}_2$, then ${\bep{}}_1 {\bep{}}_2 = 0$.
  
		Finally, any multicomplex number $\eta \in \mM \mC (n)$ can be rewritten as
		\begin{align*}
		\eta = \sum_{\bep{} \in \mathcal{E}(k, n)} \eta_{\bep{}} \bep{} ,
		\end{align*}
	where $\eta_{\bep{}} \in \mM\mC (k-1)$ for all $\bep{} \in \mathcal{E}(k, n)$. The special case when $k = 2$ will be of particular importance to us. For this reason, we let $\mathcal{E}_n := \mathcal{E} (2, n)$, and therefore any $\eta \in \mM \mC (n)$ can be written as
    \begin{equation}  
        \eta = \sum_{\bep{} \in \mathcal{E}_n} \eta_{\bep{}} \bep{} ,    \label{Eq:IdempReprComplex}
    \end{equation}
    where $\eta_{\bep{}} \in \mM \mC (1)$ for all $\bep{} \in \mathcal{E}_n$.
	
	These new idempotent representations still have the advantage of simplifying the operation of multiplication. If
		\begin{align*}
		\eta = \sum_{\bep{} \in \mathcal{E} (k, n)} \eta_{\bep{}} \bep{} \quad \text{and} \quad \zeta = \sum_{\bep{} \in \mathcal{E}(k, n)} \zeta_{\bep{}} \bep{},
		\end{align*}
	then the following holds:
		\begin{enumerate}[label=\arabic*)]
		\item $\eta = \zeta$ if and only if $\eta_{\bep{}} = \zeta_{\bep{}}$ for all $\bep{} \in \mathcal{E} (k, n)$;
		\item $\eta + \zeta = \sum_{\bep{} \in \mathcal{E} (k, n)} (\eta_{\bep{}} + \zeta_{\bep{}}) \bep{}$;
		\item $\eta \zeta = \sum_{\bep{} \in \mathcal{E} (k , n)} (\eta_{\bep{}} \zeta_{\bep{}}) \bep{}$.
		\end{enumerate}
When $n = 1$, we set $\mathcal{E}_1 := \{1 \}$ so that the representation in terms of the set $\mathcal{E}_n$ is still valid with coefficients in $\mM \mC (1)$. 

\subsection{Topological considerations} We end this subsection with some topological concepts. 
Under the identification $\mM \mC(1) \cong \mC$, we
write $| \cdot |$ for the modulus on $\mM \mC(1)$, and we equip $\mM \mC(n)$
with the norm $\| \eta \| := \max_{\bep{} \in \mathcal{E}_n} | \eta_{\bep{}} |$; since all
norms on a finite-dimensional real vector space are equivalent, the limits
considered in the next subsection do not depend on this choice. In particular, $h \ra 0$ in
$\mM \mC(n)$ if and only if $h_{\bep{}} \ra 0$ for every $\bep{} \in \mathcal{E}_n$. 

For $a \in \mM \mC(n)$ and $r > 0$, the ball
    \[
    P ( a , r ) := \{ \eta \in \mM \mC(n)\, :\, \| \eta - a \| < r \}
    = \{ \eta \in \mM \mC(n)\, :\, | \eta_{\bep{}} - a_{\bep{}} | < r , \ \bep{} \in \mathcal{E}_n \}
    \]
is an \emph{idempotent polydisc}; these sets are open and convex, and they
form a basis of neighborhoods of $a$. Finally, given
$F : U \ra \mM \mC(n)$, we write $F_{\bep{}} (\eta ) := \pi_{\bep{}}(F (\eta ))$ for
the idempotent component of $F$, where $\pi_{\bep{}} : \mM \mC (n) \ra \mM \mC (1)$ is the projection on the idempotent component of $\bep{}$, for $\bep{} \in \mathcal{E}_n$. We then have $F_{\bep{}} : \mM \mC (n) \ra \mM \mC (1)$ for every $\bep{} \in \mathcal{E}_n$ and we can then write $F = \sum_{\bep{} \in \mathcal{E}_n} F_{\bep{}} \bep{}$. We also denote by $\iota_{\bep{}} : \mM \mC (1) \ra \mM \mC (n)$ the inclusion map $z \mapsto z \bep{}$. The maps $\pi_{\bep{}}$ and $\iota_{\bep{}}$ are both linear maps and therefore real-differentiable, for any $\bep{} \in \mathcal{E}_n$. 

We record in the following lemma the characterization of invertibility in idempotent
coordinates. The result is classical (see \cite{Baley}, and
\cite{Bicomplex, RochonShapiro} for the bicomplex case), and we include the
short proof for convenience. 

    \begin{lemma}\label{L:Invertibility}
    A multicomplex number $\eta \in \mM \mC(n)$ is invertible if and only if
    $\eta_{\bep{}} \neq 0$ for every $\bep{} \in \mathcal{E}_n$, in which case
        \[
        \eta^{-1} = \sum_{\bep{} \in \mathcal{E}_n} \eta_{\bep{}}^{-1}\, \bep{} .
        \]
    Moreover, the group $\mM \mC(n)^\times$ of invertible elements is open and
    dense in $\mM \mC(n)$.
    \end{lemma}

    \begin{proof}
    Since multiplication is componentwise in the representation \eqref{Eq:IdempReprComplex} and since $1 = \sum_{\bep{} \in \mathcal{E}_n} \bep{}$, we have $\eta \zeta = 1$ if and only if $\eta_{\bep{}} \zeta_{\bep{}} = 1$ for every $\bep{} \in \mathcal{E}_n$. 
    As $\mM \mC(1)$ is a field, this proves the criterion and the formula for the inverse. 
    The map $\eta \mapsto ( \eta_{\bep{}} \, : \, \bep{} \in \mathcal{E}_n )$ is a real-linear bijection from $\mM \mC(n)$ onto $\mM \mC(1)^{2^{n-1}}$, hence a homeomorphism, and $\mM \mC(n)^\times$ is the preimage of the set of $2^{n-1}$-tuples with nonzero entries, which is open and dense in $\mM \mC(1)^{2^{n-1}}$.
    \end{proof}

The complement of $\mM \mC(n)^\times$, often called the \emph{null cone} \cite{CourchesneTremblay2025}, is the finite union of the linear subspaces $\{ \eta : \eta_{\bep{}} = 0 \}$, for $\bep{} \in \mathcal{E}_n$, each of real codimension $2$. 

\subsection{Multicomplex holomorphy}\label{SS:MCHolo}
Let $U \subseteq \mM \mC(n)$ be an open set. Throughout this section, a function $F : U \ra \mM \mC(n)$ is said to be of class $C^1$ if it is continuously differentiable as a map between real vector spaces, and we denote by $D F (\eta ) : \mM \mC(n) \ra \mM \mC(n)$ its (real) differential at the point $\eta \in U$. 
We adopt the following differential formulation of holomorphy.

    \begin{definition}\label{D:MCHolo}
    A function $F : U \ra \mM \mC(n)$ of class $C^1$ is \emph{multicomplex holomorphic} on $U$ if, for every $\eta \in U$, the differential $D F (\eta)$ is $\mM \mC(n)$-linear, that is, if there exists $\lambda (\eta ) \in \mM \mC(n)$ such that
        \begin{align*}
        DF (\eta ) [h] = \lambda (\eta )\, h , \qquad h \in \mM \mC(n).
        \end{align*}
    We denote by $\hol(U)$ the set of multicomplex holomorphic functions on $U$.
    \end{definition}

A large part of the literature on bicomplex and multicomplex function theory
defines holomorphy through a difference quotient rather than through the
differential: a function $F$ is said to be \emph{derivable} at $\eta_0$ if the
quotient $( F (\eta ) - F (\eta_0 ) ) ( \eta - \eta_0 )^{-1}$ admits a limit as
$\eta \ra \eta_0$ with $\eta - \eta_0$ invertible (for the multicomplex case, see \cite{Baley} and, for a detailed account of the bicomplex case, see \cite{Bicomplex}). Then $F$ is a holomorphic function on an open set $U \subseteq \mM \mC (n)$ if it is derivable for every $\eta \in U$. 

In the proposition below, we reconcile the two
points of view: on an open set, 
both formulations are
equivalent to a local separation of variables in the idempotent coordinates. By Lemma~\ref{L:Invertibility}, every punctured ball centered at the origin meets $\mM \mC(n)^\times$, so the limit in condition (ii) below, taken along invertible increments, is meaningful. We emphasize that no regularity whatsoever is assumed on $F$ in condition (ii): continuity of $F$ is part of the conclusion.

    \begin{proposition}\label{P:DerivativeEquivalence}
    Let $U \subseteq \mM \mC(n)$ be open and let $F : U \ra \mM \mC(n)$ be a
    function. The following statements are equivalent.
        \begin{enumerate}[label=(\roman*)]
        \item $F$ is multicomplex holomorphic on $U$ in the sense of
        Definition~\ref{D:MCHolo}.
        \item For every $\eta_0 \in U$, the limit
            \begin{equation}\label{Eq:MCDerivative}
            F' (\eta_0 ) := \lim_{\substack{h \ra 0 \\ h \in \mM \mC(n)^\times}}
            \frac{ F (\eta_0 + h ) - F (\eta_0 )}{h} 
            \end{equation}
        exists in $\mM \mC(n)$, where $h$ ranges over the invertible elements
        such that $\eta_0 + h \in U$.
        \item Every point of $U$ admits an idempotent polydisc neighborhood
        $P ( a , r ) \subseteq U$ on which
            \begin{equation}\label{Eq:LocalSplitting}
            F (\eta ) = \sum_{\bep{} \in \mathcal{E}_n } \varphi_{\bep{}} ( \eta_{\bep{}} )\, \bep{} ,
            \qquad \eta \in P ( a , r ),
            \end{equation}
        where each $\varphi_{\bep{}}$ is a holomorphic function of one complex
        variable on the disc
        $D_{\bep{}} := \{ z \in \mM \mC(1) : | z - a_{\bep{}} | < r \}$.
        \end{enumerate}
    In this case, $F$ is of class $C^\infty$ on $U$ and, for every
    $\eta \in U$,
        \begin{equation}\label{Eq:DerivativeIdentification}
        F' (\eta ) = \lambda (\eta ) = DF (\eta ) [ 1 ]
        = \sum_{\bep{} \in \mathcal{E}_n} \varphi_{\bep{}} ' ( \eta_{\bep{}} )\, \bep{}
        \end{equation}
    in the local representation \eqref{Eq:LocalSplitting}.
    \end{proposition}

    \begin{proof}
    We prove that (iii) implies (i) and (ii), that (i) implies (iii), and
    finally that (ii) implies (iii). The identities in
    \eqref{Eq:DerivativeIdentification} are established along the way.

    (iii) $\Ra$ (i). Suppose that \eqref{Eq:LocalSplitting} holds on
    $P = P ( a , r ) \subseteq U$. 
    For each $\bep{} \in \mathcal{E}_n$, define $\pi_{\bep{}} : \mM \mC (n) \ra \mM \mC (1)$ by $\pi_{\bep{}} (\eta) := \eta_{\bep{}}$ for $\eta = \sum_{\bep{} \in \mathcal{E}_n} \eta_{\bep{}} \bep{}$ and $\iota_{\bep{}} : \mM \mC (1) \ra \mM \mC (n)$ by $\iota_{\bep{}} (z) := z \bep{}$, so that $F = \sum_{\bep{} \in \mathcal{E}_n} G_{\bep{}}$ on $P$, where $G_{\bep{}} := \iota_{\bep{}} \circ \varphi_{\bep{}} \circ \pi_{\bep{}}$. 
    Each map $\pi_{\bep{}}$ and $\iota_{\bep{}}$ is real-linear, and each $\varphi_{\bep{}}$ is holomorphic, hence of class $C^\infty$. Consequently $F$ is of class $C^\infty$ on $P$ and therefore on $U$, since $U$ is covered by such polydiscs. 
    Moreover, by the chain rule and the complex-linear nature of the differential of a holomorphic complex-valued function, 
    $$
        DG_{\bep{}} (\eta) [h] = \iota_{\bep{}} [ D\varphi_{\bep{}} (\eta_{\bep{}}) [ \pi_{\bep{}} [h]]] = \varphi_{\bep{}}' (\eta_{\bep{}}) h_{\bep{}} \bep{} ,
    $$
    and therefore
        \[
        DF (\eta ) [ h ]
        = \sum_{\bep{} \in \mathcal{E}_n}\varphi_{\bep{}} ' ( \eta_{\bep{}} )\, h_{\bep{}}\, \bep{}
        = \Big( \sum_{\bep{} \in \mathcal{E}_n} \varphi_{\bep{}} ' ( \eta_{\bep{}} )\, \bep{} \Big)\, h ,
        \qquad h \in \mM \mC(n) ,
        \]
    the second equality is because multiplication is componentwise. Thus
    $DF (\eta )$ is $\mM \mC(n)$-linear with
    $\lambda (\eta ) = \sum_{\bep{} \in \mathcal{E}_n} \varphi_{\bep{}} ' ( \eta_{\bep{}} ) \bep{}$, so
    $F \in \hol (U)$, and evaluating at $h = 1$ gives
    $\lambda (\eta ) = DF (\eta ) [ 1 ]$.

    (iii) $\Ra$ (ii). Fix $\eta_0 \in U$. Choose the polydisc in (iii) centered at $a = \eta_0$. Let $h \in \mM \mC(n)^\times$
    with $\eta_0 + h \in P ( a , r )$. By Lemma~\ref{L:Invertibility},
    $h_{\bep{}} \neq 0$ for every $\bep{}$ and inversion is componentwise, so that
        \[
        \frac{F (\eta_0 + h ) - F (\eta_0 )}{h} 
        = \sum_{\bep{} \in \mathcal{E}_n}
        \frac{ \varphi_{\bep{}} ( a_{\bep{}} + h_{\bep{}} ) - \varphi_{\bep{}} ( a_{\bep{}} ) }{ h_{\bep{}} }\,
        \bep{} .
        \]
    As $h \ra 0$, each component $h_{\bep{}}$ tends to $0$ while remaining nonzero,
    and since $\varphi_{\bep{}}$ is complex-differentiable at $a_{\bep{}}$, the ${\bep{}}$-th
    component of the right-hand side converges to $\varphi_{\bep{}} ' ( a_{\bep{}} )$.
    componentwise convergence being equivalent to convergence in
    $\mM \mC(n)$, the limit \eqref{Eq:MCDerivative} exists and equals
    $\sum_{\bep{} \in \mathcal{E}_n} \varphi_{\bep{}} ' ( a_{\bep{}} ) \bep{} = F'(\eta_0)$. This proves
    (ii).

    (i) $\Ra$ (iii). Assume that $F$ is of class $C^1$ on $U$ and that $DF (\eta )$ is multicomplex linear for any $\eta \in U$. Let $a \in U$ and choose $r > 0$ with $P := P ( a , r ) \subseteq U$. Writing $F = \sum_{\bep{} \in \mathcal{E}_n} F_{\bep{}} \bep{}$ and $\lambda = \sum_{\bep{} \in \mathcal{E}_n} \lambda_{\bep{}} \bep{}$, we see that $F_{\bep{}} = \pi_{\bep{}} \circ F$ and $\lambda_{\bep{}} = \pi_{\bep{}} \circ \lambda$. Therefore, from the chain rule, we obtain
        \begin{equation}\label{Eq:ComponentwiseDifferential}
        D F_{\bep{}} (\eta ) [ h ] = \pi_{\bep{}} [DF (\eta ) [h] ] = \pi_{\bep{}} [\lambda (\eta ) h] = \lambda_{\bep{}} (\eta )\, h_{\bep{}} ,
        \end{equation}
    for $\eta \in P$, $h \in \mM \mC(n)$, and $ \bep{} \in \mathcal{E}_n$. 

    Fix $\bep{}$ and let $\bep{}' \neq \bep{}$. If $\eta , \zeta \in P$ differ only in
    their $\bep{}'$ components, the segment joining them lies in $P$ and is of
    the form $t \mapsto \eta + t\, w\, \bep{}'$, $t \in [ 0 , 1 ]$, with $w = \zeta_{\bep{}'} - \eta_{\bep{}'} \in \mM \mC (1)$.
    By \eqref{Eq:ComponentwiseDifferential}, the derivative of the $C^1$ map $t \mapsto F_{\bep{}} ( \eta + t\, w\, \bep{}' )$
    equals $D F_{\bep{}} ( \cdot ) [ w \bep{}' ] = \lambda_{\bep{}} ( \cdot )\,
    \pi_{\bep{}} ( w \bep{}' ) = 0$, since $\pi_{\bep{}} ( w \bep{}' ) = 0$ for $\bep{} \neq \bep{}'$. Hence
    $F_{\bep{}} ( \zeta ) = F_{\bep{}} ( \eta )$ by the mean value theorem applied to the real components of $F_{\bep{}}$. Since any two points of $P$ with the
    same $\bep{}$ component differ by finitely many such one-component moves
    inside $P$, the component $F_{\bep{}}$ depends only on $\eta_{\bep{}}$ on $P$. Therefore there is
    a function $\varphi_{\bep{}} : D_{\bep{}} \ra \mM \mC(1)$ with
    $F_{\bep{}} (\eta ) = \varphi_{\bep{}} ( \eta_{\bep{}} )$ for $\eta \in P$.

    It remains to check that each $\varphi_{\bep{}}$ is holomorphic on $D_{\bep{}}$. Let
    $z \in D_{\bep{}}$ and set $\eta := a + ( z - a_{\bep{}} )\, \bep{} \in P$. The map
    $z ' \mapsto \varphi_{\bep{}} ( z ' ) = F_{\bep{}} \big( \eta + ( z ' - z ) \bep{} \big)$
    is the composition of a real-affine map with $F_{\bep{}}$, hence is differentiable at $z$, and by \eqref{Eq:ComponentwiseDifferential} its
    differential is $w \mapsto D F_{\bep{}} (\eta ) [ w \bep{} ] =
    \lambda_{\bep{}} (\eta )\, w$, which is $\mM \mC(1)$-linear. Thus $\varphi_{\bep{}}$ is
    complex-differentiable at every point of $D_{\bep{}}$, hence holomorphic on
    $D_{\bep{}}$. This proves (iii).

    (ii) $\Ra$ (iii). This is the only implication requiring some care, since
    the increments in \eqref{Eq:MCDerivative} must avoid the null cone and no
    regularity of $F$ is assumed. Let $a \in U$ and choose $r > 0$ so that
    $P := P ( a , r ) \subseteq U$.

    We first record the estimate provided by (ii). Let $q \in P$. Applying
    \eqref{Eq:MCDerivative} at $q$, there exists
    $\delta_q > 0$, which we may take small enough that
    $\| h \| < \delta_q$ implies $q + h \in U$, such that for $h \in \mM \mC(n)^\times$, $0 < \| h \| < \delta_q$,
        \begin{equation}\label{Eq:BasicEstimate}
        \big\| \big( F ( q + h ) - F ( q ) \big)\, h^{-1} - F ' ( q ) \big\|
        \leq 1 .
        \end{equation}
    For such $h$, write $Q := ( F ( q + h ) - F ( q ) )\, h^{-1}$, so that
    $F ( q + h ) - F ( q ) = Q\, h$ exactly. Taking the $\bep{}$ component and
    using $| Q_{\bep{}} | \leq \| F ' ( q ) \| + 1$, which follows from
    \eqref{Eq:BasicEstimate}, we deduce
        \begin{equation}\label{Eq:ComponentEstimate}
        \big| F_{\bep{}} ( q + h ) - F_{\bep{}} ( q ) \big| \leq C_q\, | h_{\bep{}} | ,
        \qquad C_q := \| F ' ( q ) \| + 1 , \quad \bep{} \in \mathcal{E}_n.
        \end{equation}
    We now show our main goal in two steps: (1) we first show that, on $P$, each component $F_{\bep{}}$ depends only on $\eta_{\bep{}}$; and (2) we then show that each $\varphi_{\bep{}}$ is holomorphic on $D_{\bep{}}$, which will complete the proof of the equivalence.

    (1) Fix $\bep{} \neq \bep{}'$ and let $\eta , \zeta \in P$ differ only in their
    $\bep{}'$ component, say $\zeta = \eta + w\, \bep{}'$ with
    $w \in \mM \mC(1) \setminus \{ 0 \}$. For $t \in [ 0 , 1 ]$, set
    $q_t := \eta + t\, w\, \bep{}' \in P$, and abbreviate
    $\delta_t := \delta_{q_t}$ and $C_t := C_{q_t}$. We claim that
    $t \mapsto F_{\bep{}} ( q_t )$ is locally constant on $[ 0 , 1 ]$, meaning that there is a neighborhood around any point in $[0, 1]$ on which $t \mapsto F_{\bep{}} (q_t)$ is constant.

    Fix $t \in [ 0 , 1 ]$ and let $t ' \in [ 0 , 1 ]$ with
    $0 < | t ' - t |\, | w | < \delta_t / 2$. Let $s$ be a real number with
        \[
        0 < s < \min \big( \delta_t / 4 ,\ \delta_{t '} ,\
        | t ' - t |\, | w | \big) ,
        \]
    and consider the increment and the auxiliary point
        \[
        h := s + ( t ' - t )\, w\, \bep{}' , \qquad y := q_t + h .
        \]
    The components of $h$ are $h_{\bep{}''} = s$ for $\bep{}'' \neq \bep{}'$ and
    $h_{\bep{}'} = s + ( t ' - t )\, w$, and are nonzero by the choice of $s$. Hence $h \in \mM \mC(n)^\times$ by
    Lemma~\ref{L:Invertibility}, and
    $\| h \| \leq s + | t ' - t |\, | w | < \delta_t$. On the other hand,
    since $q_{t '} = q_t + ( t ' - t )\, w\, \bep{}'$, we also have
        \[
        y = q_{t '} + s , \qquad s \in \mM \mC(n)^\times , \quad
        \| s \| = s < \delta_{t '} ,
        \]
    the real number $s = s \cdot 1$ being invertible because all of its
    components equal $s \neq 0$. Applying \eqref{Eq:ComponentEstimate} at the
    point $q_t$ with the increment $h$, whose $\bep{}$ component is $s$ because
    $\bep{} \neq \bep{}'$, and then at the point $q_{t '}$ with the increment $s$, we
    obtain
        \[
        | F_{\bep{}} ( y ) - F_{\bep{}} ( q_t ) | \leq C_t\, s
        \qquad \text{and} \qquad
        | F_{\bep{}} ( y ) - F_{\bep{}} ( q_{t '} ) | \leq C_{t '}\, s .
        \]
    Hence $| F_{\bep{}} ( q_{t '} ) - F_{\bep{}} ( q_t ) | \leq ( C_t + C_{t '} )\, s$ for
    every sufficiently small $s > 0$, and therefore
    $F_{\bep{}} ( q_{t '} ) = F_{\bep{}} ( q_t )$. This holds in the neighborhood $I_t := (t - \rho_t , t + \rho_t) \cap [0, 1]$ where $\rho_t := \delta_t / (2 |w| )$. Hence, $t\mapsto F_{\bep{}}(q_{t'})$ is constant on $I_t$ and this proves what we wanted. Now, since $[0, 1]$ is a connected set and $t \mapsto F_{\bep{}} ( q_t )$ is locally constant, it must be constant on $[0, 1]$. Therefore, we obtain $F_{\bep{}} ( \zeta ) = F_{\bep{}} ( \eta )$. As in the proof of (i) $\Ra$ (iii), it follows that $F_{\bep{}}$ depends only on the component $\eta_{\bep{}}$ and hence $F_{\bep{}} ( \eta ) = \varphi_{\bep{}} ( \eta_{\bep{}} )$ in $P$ for some function $\varphi_{\bep{}} : D_{\bep{}} \ra \mM \mC(1)$.

    (2) Let $z \in D_{\bep{}}$
    and set $\eta := a + ( z - a_{\bep{}} )\, \bep{} \in P$. Let $( z_p )_{p \geq 1}$
    be any sequence in $D_{\bep{}} \setminus \{ z \}$ with $z_p \ra z$, and set
    $h_p := z_p - z \in \mM \mC(1) \setminus \{ 0 \} \subseteq \mM \mC (n)$. All the idempotent
    components of $h_p$ equal $z_p - z \neq 0$, so
    $h_p \in \mM \mC(n)^\times$ by Lemma~\ref{L:Invertibility}, while
    $\| h_p \| = | z_p - z | \ra 0$ and $\eta + h_p \in P$ for $p$ large
    enough. By (ii) applied at $\eta$, the quotients
    $( F ( \eta + h_p ) - F ( \eta ) )\, h_p^{-1}$ converge to $F ' (\eta )$. Applying $\pi_{\bep{}}$ and using Step (1), we obtain
        \[
        \frac{ \varphi_{\bep{}} ( z_p ) - \varphi_{\bep{}} ( z ) }{ z_p - z }
        \; \longrightarrow \; \pi_{\bep{}} \big( F ' ( \eta ) \big).
        \]
    Since the sequence $( z_p )$ was arbitrary, $\varphi_{\bep{}}$ is
    complex-differentiable at $z$, with
    $\varphi_{\bep{}} ' ( z ) = \pi_{\bep{}} \big( F ' ( \eta ) \big)$. Thus $\varphi_{\bep{}}$ is
    complex-differentiable at every point of the open disk $D_{\bep{}}$, hence
    holomorphic on $D_{\bep{}}$. This proves (iii).
    
    The identities in \eqref{Eq:DerivativeIdentification} were
    obtained along the way.
    \end{proof}

    \begin{remark}
    (1) The equivalence is local in an essential way: for $n \geq 2$, the $\mM \mC(n)$-linearity of the differential at a \emph{single} point does not imply the existence of the limit \eqref{Eq:MCDerivative} at that point. For instance, for $n = 2$, consider $G : \mM \mC(2) \ra \mM \mC(2)$ defined by $G (\eta ) := ( \eta_{\bec{2}} )^2\, \be{2}$. Then $G$ is a polynomial map of the underlying real coordinates with $DG (\eta ) [ h ] = 2\, \eta_{\bec{2}}\, h_{\bec{2}}\, \be{2}$; in
    particular $DG ( 0 ) = 0$, so that $DG ( 0 )$ is $\mM \mC(2)$-linear. Nevertheless, for the invertible increments
    $h_t := t^3\, \be{2} + t\, \bec{2}$ with $t > 0$,
        \[
        \big( G ( h_t ) - G ( 0 ) \big)\, h_t^{-1}
        = t^2 \cdot t^{-3}\, \be{2} = t^{-1}\, \be{2} ,
        \]
    which is unbounded as $t \ra 0$, while along the real increments $g_t := t$ the same quotient equals $t\, \be{2} \ra 0$; hence the limit \eqref{Eq:MCDerivative} does not exist at $0$. This is consistent with Proposition~\ref{P:DerivativeEquivalence}, since $DG (\eta )$ fails to be $\mM \mC(2)$-linear at every $\eta$ with $\eta_{\bec{2}} \neq 0$, so that condition (i) holds on no neighborhood of $0$. The underlying mechanism is metric: for $n \geq 2$, the norm $\| h^{-1} \|$ may be much larger than $\| h \|^{-1}$ when $h$ approaches $0$ near the null cone, whereas $| h^{-1} | = | h |^{-1}$ when $n = 1$. A comment about this phenomenon was made for bicomplex-valued holomorphic functions in \cite{Bicomplex}.
    
    (2) For the remainder of the paper, all mappings under consideration are assumed to be of class $C^1$, unless stated otherwise.
    \end{remark}

From the canonical representation \eqref{Eq:CartesianMonomialF2n}, we write $\eta = \sum_{\vec b \in \mathbb F_2^n} x_{\vec b}\, \mathbf i^{\vec b}$, where $x_{\vec b} := \eta_{\mathbf i^{\vec b}} \in \mR$, so that the functions $x_{\vec b}$ form a system of real coordinates on $\mM \mC(n)$. Definition~\ref{D:MCHolo} is then equivalent to a system of equations of Cauchy--Riemann type. We include a proof of this result for the convenience of the reader. 
Recall that the symbol $\oplus$ is used to denote the sum of vectors in $\mathbb{F}_2^n$. 

    \begin{proposition}\label{P:MCHoloSystem}
    Let $F : U \ra \mM \mC(n)$ be of class $C^1$. Then $F$ is multicomplex holomorphic on $U$ if and only if
        \begin{equation}\label{Eq:MCCRSystem}
        \frac{\partial F}{\partial x_{\vec b \oplus \vec e_j}} = \im{j} \frac{\partial F}{\partial x_{\vec b}} \qquad (1 \leq j \leq n , \ \vec b \in \mathbb F_2^n , \ b_j = 0),
        \end{equation}
    on $U$. In particular, we have $\lambda = \partial F / \partial x_{\vec 0}$ on $U$.
    \end{proposition}

    \begin{proof}
    Observe first that $\partial F / \partial x_{\vec b} (\eta ) = DF(\eta) [\mathbf i^{\vec b}]$ and that, when $b_j = 0$, we have $\mathbf i^{\vec b \oplus \vec e_j} = \im{j} \mathbf i^{\vec b}$ without any sign correction, by \eqref{Eq:MonomialMult}. If $F$ is multicomplex holomorphic, then
    \[
    DF(\eta)[\mathbf i^{\vec b \oplus \vec e_j}] = \lambda (\eta)\, \im{j} \mathbf i^{\vec b} = \im{j}\, \lambda (\eta )\, \mathbf i^{\vec b} = \im{j}\, DF (\eta ) [\mathbf i^{\vec b}],
    \]
    which is \eqref{Eq:MCCRSystem}. Conversely, assume that \eqref{Eq:MCCRSystem} holds.  Fix $\eta \in U$. Let $\vec{b} \in \mathbb{F}_2^n$ with $\vec{b} \neq 0$ and choose $k$ such that $b_k = 1$. From \eqref{Eq:MCCRSystem}, we have
    $$
        DF (\eta ) [\im{k}] = \frac{\partial F}{\partial x_{\vec{0} \oplus \vec{e}_k}} (\eta ) = \im{k} \frac{\partial F}{\partial x_{\vec{0}}} (\eta ) = \im{k} DF (\eta ) [1] 
    $$
    and, by induction on the nonzero entries of the vector $\vec{b}$ to build the monomial $\im{}^{\vec{b}}$, we obtain
    \begin{equation}
    DF(\eta) [\mathbf i^{\vec b}] = \frac{\partial F}{\partial x_{\vec{b}}} (\eta ) = \im{}^b \frac{\partial F}{\partial x_{\vec{0}}} (\eta ) = \mathbf i^{\vec b}\, DF(\eta)[1] . \label{Eq:CAuchyRiemannForEVeryI}
    \end{equation}
    Since $DF(\eta)$ is real-linear and the monomials $\mathbf i^{\vec b}$ form a real basis of $\mM \mC(n)$, it follows that $DF (\eta) [h] = h\, DF(\eta)[1]$ for every $h \in \mM \mC(n)$, so $F$ is multicomplex holomorphic with $\lambda (\eta ) = DF (\eta ) [1] = \partial F/\partial x_{\vec 0} (\eta )$.
    \end{proof}

For $n = 1$, the system \eqref{Eq:MCCRSystem} consists of the single equation $\partial F / \partial x_{\vec e_1} = \im{1}\, \partial F / \partial x_{\vec 0}$, which is the classical Cauchy--Riemann equation written in complex form, while for $n = 2$, from \eqref{Eq:CAuchyRiemannForEVeryI}, it consists of the four bicomplex Cauchy--Riemann equations (see \cite[Theorem 7.3.1]{Bicomplex}). Observe that the proof of Proposition~\ref{P:MCHoloSystem} only uses the fact that the units $\im{1} , \ldots , \im{n}$ commute, square to $-1$, and generate $\mM \mC(n)$ as a real algebra: the same characterization, therefore, holds with respect to any ordered  $n$-tuple of commuting roots of $-1$ that generates $\mM \mC (n)$ as a real algebra. 


\section{Involutions preserving elementary units}\label{S:InvoPreserveElemUnits}

We start by giving a precise definition of what we mean by an involution on an associative real algebra $A$ with multiplicative identity.
        \begin{definition}\label{D:Involution}
        A function $f : A \ra A$ is said to be an \textit{involution} if the following conditions are satisfied:
            \begin{enumerate}[label=(\alph*)]
            \item $f (f (\eta )) = \eta$ for any $\eta \in A$;
            \item $f (\eta + \zeta ) = f (\eta ) + f (\zeta )$ and $f (\lambda \eta ) = \lambda f(\eta )$ for any $\eta , \zeta \in  A$ and $\lambda \in \mR $;
            \item $f (\eta \zeta ) = f (\eta ) f (\zeta )$ for any $\eta, \zeta \in  A$.
            \end{enumerate}
        \end{definition}
By definition, an involution is a real-algebra automorphism that is its own inverse. When $f$ is invertible but is not assumed to satisfy $f^2=\mathrm{Id}$, we simply call $f$ a real-algebra automorphism of $A$.

When $A$ is the division algebra of quaternions, we know from \cite{SangwineTodd2007, Lawson2021} that there are infinitely many involutions. If $q = a + b \im{} + c \jm{} + d \km{}$ is a quaternion with the usual rules
    \[
        \im{} \jm{} = -\jm{} \im{} = \km{} , \quad \jm{} \km{} = -\km{} \jm{} = \im{} , \quad \km{} \im{} = -\im{} \km{} = \jm{},
    \]
and $\ims{} = \jms{} = \kms{} = -1$, then any nontrivial involution is given by $f_{\mu} (q) = -\mu q \mu$, where $\mu = a_0 \im{} + b_0 \jm{} + c_0 \km{}$ with $a_0^2 + b_0^2 + c_0^2 = 1$. For other real algebras, however, the situation might change drastically.

In a previous note \cite{Parise}, the second author replaced the quaternions by the commutative ring of bicomplex numbers. From \cite[Theorem 1]{Parise}, we know that there are exactly six involutions of $\mM \mC(2)$. This result contrasts with the similar one obtained for the quaternions and therefore makes the set of bicomplex numbers akin to the complex numbers, where the only involutions are $z \mapsto z$ and $z \mapsto \Bar{z}$.

We extend the classification from \cite{Parise} to the multicomplex numbers of order $n \geq 1$. However, as mentioned in the introduction, we restrict our attention to $\mI(n)$-preserving involutions of $\mM \mC(n)$, $n \geq 1$. Recall that an $\mI(n)$-preserving involution is an involution $f$ of $\mM \mC(n)$ that maps every element of $\mI(n)$ into $\pm\mI(n)$. We therefore obtain
    \begin{theorem}\label{T:InvolutionsPreservingBasisUnits}
    The number of $\mI(n)$-preserving involutions of $\mM \mC(n)$, $n \geq 1$, is
        \begin{align*}
            \sum_{k=\lceil n/2\rceil}^{n}  \Big( \prod_{j = 1}^{k-1} \frac{2^n - 2^j}{2^k - 2^j} \Big) \Big( \prod_{j = 0}^{n-k-1} (2^k - 2^j) \Big) 2^k,
        \end{align*}
    where an empty product is understood to be equal to $1$.
    \end{theorem}

\subsection{Proof of Theorem~\ref{T:InvolutionsPreservingBasisUnits}}
If $f$ is an involution, then for any multicomplex number $\eta$, we have
        \begin{align*}
        f ( \eta ) = \sum_{\vec{b} \in \mathbb{F}_2^n} \eta_{\vec{b}} f (\im{}^{\vec{b}}).
        \end{align*}
    Since every $\im{}^{\vec{b}}\in\mI(n)\setminus\{1\}$ is a product of the generators $\im{1},\ldots,\im{n}$ and $f$ is a real-algebra automorphism, its values are completely determined by its action on $\im{1} , \im{2} , \ldots , \im{n}$.
        
    Now, we have $f(\im{k})^2 = -1$, and since we restrict our attention to involutions preserving $\mI(n)$, $f(\im{k})$ must be a signed product of an odd number of imaginary units $\im{1}$, $\im{2}$, $\ldots$, $\im{n}$. Therefore, an $\mI(n)$-preserving involution $f$ of $\mM \mC(n)$ is determined by the images of the generators $\im{1}$, $\im{2}$, $\ldots$, $\im{n}$ as follows:
\[
f(\im{k})=\im{1}^{m_{1,k}}\im{2}^{m_{2,k}}\cdots \im{n}^{ m_{n,k}}(-1)^{w_{k}},\qquad\, 1\le k\le n,
\]
where $m_{\ell, k}, w_k \in \{ 0 , 1 \}$ for each $1 \leq \ell \leq n$ and $1 \leq k \leq n$. Furthermore, for $1\le k \le n$ we have
\[
-1=f(-1)=f(\ims{k})=f(\im{k})^2=(-1)^{m_{1,k}}(-1)^{m_{2,k}}\cdots (-1)^{m_{n,k}},
\]
which implies 
\[
\sum_{\ell=1}^n m_{\ell, k}\equiv 1\pmod{2},\qquad 1\le k\le n. 
\]
Consequently, the problem reduces to counting the exponent matrices and sign vectors for which these assignments extend to involutions. 

Now, given an $n \times n$ matrix $M$ with entries in $\{0,1\}$ and columns $\vec{m}_1,\ldots,\vec{m}_n$, and given a vector $\vec w=(w_1,\ldots,w_n)^\top\in\mathbb F_2^n$, we denote by $f_{M,\vec w}$ the map defined on $1$ and the imaginary units $\im{1}$, $\im{2}$, $\ldots$, $\im{n}$ by $f_{M, \vec w} (1) := 1$ and
\[
f_{M,\vec w}(\im{k}):=(-1)^{w_k}\,\mathbf i^{\vec{m}_k},\qquad 1\le k\le n.
\]
The assignment defining $f_{M,\vec w}$ extends uniquely to a unital real-algebra endomorphism of $\mM \mC(n)$ precisely when the image of each generator squares to $-1$, equivalently when every column of $M$ has odd sum, which is the parity requirement above. Once extended, it is an automorphism if and only if $M$ is invertible over $\mathbb F_2$. In particular, this is automatic once $M^2=I$. In all cases, $f_{M,\vec w}$ maps the basis element $\mathbf i^{\vec b}$ to $\pm\, \mathbf i^{M \vec b}$. We therefore define a function $s : \mathbb F_2^n \ra \mathbb F_2$, called the \emph{sign function} of $f_{M,\vec w}$, by the relation
\begin{equation}\label{Eq:SignFunction}
f_{M,\vec w}(\mathbf i^{\vec b})=(-1)^{s(\vec b)}\,\mathbf i^{M\vec b}, \qquad \vec b \in \mathbb F_2^n.
\end{equation}
Here $M\vec b$ is computed modulo $2$.

    \begin{lemma}[Sign Lemma]\label{L:SignVectors}
    Let $M$ be an $n\times n$ matrix with entries in $\{0,1\}$ whose columns $\vec{m}_1,\ldots,\vec{m}_n$ all have odd sum, let $\vec w\in\mathbb F_2^n$, and set $Y:=M-I$, understood modulo $2$. Then the following statements hold.
        \begin{enumerate}[label=(\alph*)]
        \item For any $\vec b \in \mathbb F_2^n$, we have $s(\vec b)=\vec b^{\top}\vec w+q(\vec b)$, where
        \[
        q(\vec b):= \sum_{k = 1}^n \lfloor t_k / 2 \rfloor \pmod 2 
        \]
        and $t_k := \sum_{l = 1}^n m_{k, l} b_l \in \mathbb{Z}$.
                \item The map $f_{M, \vec{w}}$ is an involution of $\mM \mC(n)$ if and only if,
            \begin{equation}\label{Eq:AffineSign}
           M^2\equiv I\pmod 2
            \quad\text{and}\quad
            Y^{\top}\vec w\equiv \vec{c} \pmod 2,
            \end{equation}
        where $\vec c:=(q(\vec m_1),\ldots,q(\vec m_n))^\top$ depends only on the matrix $M$.
        \item If $M^2\equiv I\pmod 2$, then the system $Y^{\top}\vec w\equiv\vec c \pmod 2$ is solvable. Consequently, the set of sign vectors $\vec w$ for which $f_{M,\vec w}$ is an involution of $\mM \mC(n)$ is a coset of $\ker(Y^{\top})$ in $\mathbb F_2^n$, and its cardinality is $2^{k}$, where $k:=\dim\ker Y$.
        \end{enumerate}
    \end{lemma}

    \begin{proof}
    To simplify the notation in the proof, let $f := f_{M,\vec w}$. 

    (a) Let $\vec{b} \in \mathbb{F}_2^n$. Since $f$ is a real-algebra
endomorphism and $\im{}^{\vec{b}} = \im{1}^{b_1} \cdots \im{n}^{b_n}$,
we obtain
    \[
        f(\im{}^{\vec{b}})
        = \prod_{j=1}^n f(\im{j})^{b_j}
        = \prod_{j=1}^n \big( (-1)^{w_j} \im{}^{\vec{m}_j} \big)^{b_j}
        = (-1)^{\vec{b}^\top \vec{w}} \prod_{j=1}^n
          \big( \im{}^{\vec{m}_j} \big)^{b_j}.
    \]
Since the imaginary units commute, we may collect the powers of each
unit $\im{k}$ in the last product. No relation $\ims{k} = -1$ is used
at this stage, so the exponent of $\im{k}$ is the integer
$\sum_{j=1}^n m_{kj} b_j = t_k$, and hence
    \[
        f(\im{}^{\vec{b}})
        = (-1)^{\vec{b}^\top \vec{w}} \prod_{k=1}^n \im{k}^{t_k}.
    \]
Write $t_k = 2 q_k + \tilde{m}_k$ with $q_k = \lfloor t_k / 2 \rfloor$
and $\tilde{m}_k \in \{0, 1\}$. From $\ims{k} = -1$, we get
$\im{k}^{t_k} = (\ims{k})^{q_k} \im{k}^{\tilde{m}_k}
= (-1)^{q_k} \im{k}^{\tilde{m}_k}$, and therefore
    \[
        f(\im{}^{\vec{b}})
        = (-1)^{\vec{b}^\top \vec{w} + \sum_{k=1}^n q_k}
          \prod_{k=1}^n \im{k}^{\tilde{m}_k}
        = (-1)^{\vec{b}^\top \vec{w} + q(\vec{b})}\, \im{}^{M \vec{b}},
    \]
where the last equality holds because
$\tilde{m}_k \equiv t_k \equiv (M\vec{b})_k \pmod{2}$, so that
$(\tilde{m}_1, \ldots, \tilde{m}_n)^\top = M\vec{b}$ in
$\mathbb{F}_2^n$. This completes the proof of (a).

    (b) Because $f$ is a unital algebra endomorphism, it is an involution if and only if $f\circ f$ fixes the generators $\im{1}, \ldots, \im{n}$. Using \eqref{Eq:SignFunction} twice together with part (a),
    \[
    f(f(\im{k})) = (-1)^{w_k} f (\mathbf i^{\vec{m}_k}) = (-1)^{w_k + s(\vec{m}_k)}\, \mathbf i^{M \vec{m}_k} = (-1)^{w_k + \vec{m}_k^{\top} \vec w + q(\vec{m}_k)}\, \mathbf i^{M^2 \vec e_k}, 
    \]
    where $\vec e_k$ denotes the $k$-th standard basis vector of $\mathbb F_2^n$, for $1 \leq k \leq n$. Therefore $f(f(\im{k})) = \im{k}$ for every $1 \leq k \leq n$ if and only if $M^2 \vec e_k \equiv \vec e_k$ and $w_k + \vec{m}_k^{\top}\vec w \equiv q (\vec{m}_k) \pmod 2$ for every $1 \leq k \leq n$. The first family of conditions is $M^2 \equiv I \pmod 2$ and, since subtraction and addition agree in $\mathbb F_2$, the second family of conditions is exactly the affine system $(I + M^{\top}) \vec w \equiv Y^{\top} \vec w \equiv \vec c \pmod 2$.

    (c) Assume that $M^2\equiv I\pmod2$, put $Y=M-I$, and set
$g:=f_{M,\vec0}$. From (a), the sign function of $g$ is $q (\vec{b})$. 

We first show that $g \circ g = f_{I, \vec{c}}$. Let $\vec e_k$ denote the $k$-th standard basis
vector, so that $\im{k} = \mathbf i^{\vec e_k}$. Since $M\vec e_k = \vec m_k$ has entries
in $\{0,1\}$, $q(\vec e_k) = 0$. Therefore
\[
g\big(g(\im{k})\big)
= g( (-1)^{q (\vec{e}_k )} \mathbf i^{\vec m_k}) = g (\im{}^{\vec{m}_k}) 
= (-1)^{q(\vec m_k)}\, \mathbf i^{M \vec m_k}
= (-1)^{c_k}\, \im{k},
\]
because $M\vec m_k = M^2 \vec e_k \equiv \vec e_k \pmod 2$ and $c_k = q(\vec m_k)$ by the
definition of $\vec c$. Thus $g \circ g$ agrees with $f_{I, \vec c}$ on the generators, and
since both are real-algebra automorphisms, $g \circ g = f_{I, \vec c}$.  

We can now show that the system $Y^{\top} \vec{w} \equiv \vec{c} \pmod 2$ has a solution. If $\vec v\in\ker Y$, then
$M\vec v=\vec v$, and the sign formula gives
\[
g(\mathbf i^{\vec v})
=(-1)^{q(\vec v)}\mathbf i^{\vec v}.
\]
Thus $(g\circ g)(\mathbf i^{\vec v})=\mathbf i^{\vec v}$. On the
other hand,
\[
(g\circ g)(\mathbf i^{\vec v})
=f_{I,\vec c}(\mathbf i^{\vec v})
=(-1)^{\vec c^\top\vec v}\mathbf i^{\vec v}.
\]
Therefore $\vec c^\top\vec v=0$ for every $\vec v\in\ker Y$ and hence $\vec{c} \in (\ker Y )^{\perp}$. As the standard bilinear form on $\mathbb F_2^n$ is
nondegenerate, $(\ker Y)^{\perp} = \operatorname{Ran}(Y^{\top})$, so $\vec c \in
\operatorname{Ran}(Y^{\top})$. This is precisely the assertion that the system $Y^{\top}
\vec w \equiv \vec c \pmod 2$ is solvable.

Finally, fix a particular solution $\vec w_0$. For any $\vec w \in \mathbb F_2^n$,
\[
Y^{\top} \vec w \equiv \vec c
\iff Y^{\top}(\vec w - \vec w_0) \equiv \vec 0
\iff \vec w - \vec w_0 \in \ker(Y^{\top}),
\]
so the solution set is the coset $\vec w_0 + \ker(Y^{\top})$. By part (b), under the
standing hypothesis $M^2 \equiv I$, these are exactly the sign vectors $\vec w$ for which
$f_{M, \vec w}$ is an involution of $\mM \mC(n)$. Hence they form a coset of
$\ker(Y^{\top})$ in $\mathbb F_2^n$. Its cardinality is $|\ker(Y^{\top})| = 2^{\dim \ker
Y^{\top}}$, and since $Y$ and $Y^{\top}$ share the same rank, and therefore the same
nullity, we have $\dim \ker Y^{\top} = \dim \ker Y = k$. The cardinality is thus $2^k$. 
    \end{proof}

With this setup, we are now ready to prove Theorem~\ref{T:InvolutionsPreservingBasisUnits}.

\begin{proof}[Proof of Theorem~\ref{T:InvolutionsPreservingBasisUnits}]
Every unital real-algebra endomorphism that maps the canonical monomial basis into signed monomials is encoded by a matrix $M$ over $\mathbb F_2$, whose columns have odd weight, together with a sign vector $\vec w\in\mathbb F_2^n$. We write $f=f_{M,\vec w}$. Set $Y:=M-I$ over $\mathbb F_2$. 

By Lemma~\ref{L:SignVectors}(b), an encoded map $f_{M,\vec w}$ is an involution if and only if every column of $M$ has odd sum, $M^2\equiv I\pmod 2$, and the affine sign condition in~\eqref{Eq:AffineSign} holds. We therefore first focus on counting the number of matrices $M$ with entries in $\{ 0, 1 \}$ such that $M^2 \equiv I \pmod 2$ and each of its columns has odd sum, and then finding the number of associated sign vectors $\vec{w}$ satisfying $Y^\top \vec{w} \equiv \vec{c} \pmod 2$. 

The equation $M^2 \equiv I \pmod 2$ is equivalent to $Y^2 \equiv 0 \pmod 2$, and since every column of $M$ has odd sum, every column of $Y$ has even sum. The problem of enumerating the matrices $M$ now becomes the problem of enumerating $n\times n$ matrices $Y$ such that
\begin{enumerate}[label=\arabic*)]
\item The entries of $Y$ are equal to 0 or 1;
    \item $Y^2\equiv 0\pmod{2}$;
    \item The sum of each column of $Y$ is $\equiv 0 \pmod{2}$.
\end{enumerate}
We denote by $k$ the dimension of the kernel of $Y$, that is, $k:= \dim(\ker (Y))$. Because $Y^2 \equiv 0 \pmod 2$, we have $\operatorname{Ran} Y \subseteq \operatorname{ker} Y$. Hence, $n - k = \operatorname{rank} Y \leq k$ and so
\begin{equation}\label{e1}
    k\ge n/2.
\end{equation}
Observe that the dimension of the kernel of $Y$ is equal to the dimension of the kernel of $Y^\top$. It will be easier to work with this transpose. 

We use the notation $\vec{e}:=(1,1,\ldots,1)^\top$.  Condition 3 on the matrix $Y$ is equivalent to 
\[
\vec{e}\in \ker (Y^\top).
\]
For a fixed value of $k$, the number of ways of choosing $\ker (Y^\top)$ with the restriction that $\vec{e}\in \ker (Y^\top)$ is given by
\begin{equation}\label{e2}
    B(k,n):=\prod_{j=1}^{k-1}\frac{2^n-2^j}{2^k-2^j},
\end{equation}
with the convention that $B(k,n)=1$ for $k=1$. To see this, first note that the number of ways of choosing an ordered sequence of $k$ linearly independent vectors (with $\vec{e}$ as the first vector of the sequence) is given by
\begin{equation}\label{e2a}
\prod_{j=1}^{k-1}(2^n-2^j),
 \end{equation}
since when choosing a new vector, one cannot choose any linear combination of previously chosen vectors. Now, many choices of vector sequences (or basis choices) will describe the same subspace.  Given a basis of linearly independent vectors, the number of ways of choosing a basis that will span the same subspace (under the condition that $\vec{e}$ is the first vector of the ordered basis) is given by 
\begin{equation}\label{e2b}
\prod_{j=1}^{k-1}(2^k-2^j).
 \end{equation}
 Equality \eqref{e2} follows from \eqref{e2a} and \eqref{e2b}.

Now, suppose that the kernel $\ker (Y^{\top})$ has been chosen. Let $k$ again be the dimension of the kernel of $Y^{\top}$. Let $\vec{u}_1,\vec{u}_2,\ldots, \vec{u}_k$ be a basis of $\ker (Y^\top)$. Let $\vec{v}_1, \vec{v}_2,\ldots,\vec{v}_{n-k}$ be vectors such that 
$\vec{u}_1,\vec{u}_2,\ldots \vec{u}_k,\vec{v}_1, \vec{v}_2,\ldots,\vec{v}_{n-k}$ is a basis of $\mathbb F_2^n$. We now count the possible linear maps $Y^\top$ with this prescribed kernel.

Since $Y^2\equiv 0\pmod 2$, we deduce that $(Y^\top)^2\equiv 0 \pmod 2$. This last identity implies that
\[
(Y^\top)^2 \vec{v}_j=\vec{0},\qquad 1\le j\le n-k,
\]
and thus
\[
Y^\top \vec{v}_j\in \ker (Y^\top ),\qquad 1\le j\le n-k.
\]
We therefore have
\[
    Y^\top\vec{v}_j=\sum_{s=1}^k r_{s,j} \vec{u}_s , 
\]
where $r_{s, j} \in \{ 0 , 1 \}$ and $1 \leq j \leq n - k$. The number of ways to choose the value of $Y^\top\vec{v}_1$ is given by
\[
2^k-1.
\]
This comes from the fact that $\vec{v}_1 \not\in \ker (Y^{\top})$ and therefore the $r_{s, j}$ cannot all be zero. One can now choose the values of $Y^\top \vec{v}_2, Y^\top\vec{v}_3,\ldots , Y^\top \vec{v}_{n - k}$ under the restriction that the vectors $Y^\top \vec{v}_j$ must be linearly independent. To see why these vectors must be linearly independent, suppose that $\vec{w}$ is a linear combination of the vectors $\vec{v}_j$ such that $Y^\top\vec{w}=\vec{0}$.  This implies $\vec{w}\in \ker (Y^\top )$.  We thus have exhibited a vector $\vec{w}$ that can be expressed both as a linear combination of the vectors $\vec{u}_j$ and as a linear combination of the vectors $\vec{v}_j$.  This contradicts the fact that  $\vec{u}_1,\vec{u}_2,\ldots \vec{u}_k,\vec{v}_1, \vec{v}_2,\ldots,\vec{v}_{n-k}$ is a basis of $\mathbb F_2^n$. This linear-independence constraint implies that the number of possible values of $Y^\top\vec{v}_j$ is $2^k-2^{j-1}$ for $1 \leq j \leq n - k$. It follows that the number of possible maps $Y^{\top}$ with the prescribed kernel is
\begin{equation}\label{e3}
D(k,n)=\prod_{j=0}^{n-k-1}(2^k-2^j).
\end{equation}

Putting everything together, the number of ways of choosing the matrix $Y^\top$, and thus $Y$, or equivalently the number of ways of choosing the matrix $M$, is given by
\[
B(k,n)D(k,n)=\prod_{j=1}^{k-1}\frac{2^n-2^j}{2^k-2^j}\prod_{j=0}^{n-k-1}(2^k-2^j).
\]

Finally, to fully specify the involution $f = f_{M, \vec w}$, one has to choose the sign vector $\vec w$. By Lemma~\ref{L:SignVectors} (c), for each admissible matrix $Y$ (or equivalently for each admissible matrix $M$), the sign vectors $\vec w$ form a coset of $\ker (Y^{\top})$ in $\mathbb F_2^n$, namely the solution set of the affine system $Y^{\top} \vec w \equiv \vec c \pmod 2$ in \eqref{Eq:AffineSign}. If the dimension of the kernel of $Y^\top$ is equal to $k$, then the number of ways of choosing the components of $\vec{w}$ is therefore equal to
\begin{equation}\label{e4}
2^k.
\end{equation}

From equations \eqref{e1}, \eqref{e2}, \eqref{e3}, and \eqref{e4}, we conclude that the number of $\mI(n)$-preserving involutions is
\[
\sum_{\lceil n/2 \rceil \le k\le n}D(k,n)B(k,n)2^k=\sum_{\lceil n/2 \rceil \le k\le n} \Big( \prod_{j = 1}^{k-1} \frac{2^n - 2^j}{2^k - 2^j} \Big) \Big( \prod_{j = 0}^{n-k-1} (2^k - 2^j) \Big) 2^k.
\]
This completes the proof.
\end{proof}

\subsection{Concluding remarks on unrestricted involutions}\label{Ss:GeneralCase}

We restricted attention to $\mI(n)$-preserving involutions because unrestricted involutions may lose their direct connection with the unit monomials and with the Cauchy--Riemann systems studied below. The unrestricted counting problem is nevertheless simple in idempotent coordinates.

Recall that any multicomplex number $\eta$ can be expressed using the idempotents from $\mathcal{E}_n$ : 
\[
\eta=\sum_{\bep{} \in \mathcal{E}_n} \eta_{\bep{}} \bep{},
\qquad \eta_{\bep{}} \in\mM \mC(1)\cong\mC .
\]
Multiplication is componentwise in this representation. Hence, as a real algebra, 
$$
\mM\mC(n) \cong \mathbb C^{2^{n-1}}.
$$
It follows that every $\mathbb R$-algebra automorphism of $\mM\mC(n)$ is obtained by permuting the $2^{n-1}$ complex components and, independently on each component, applying either the identity or complex conjugation. Thus, 
$$
\operatorname{Aut}_{\mathbb R}(\mM\mC(n)) \cong C_2 \wr S_{2^{n-1}} \cong B_{2^{n-1}},
$$
where $C_2$ is the cyclic group of order $2$, $S_n$ the symmetric group on $n$ symbols, and $B_n$ denotes the group of signed permutations of length $n$, namely permutations of $\{1, 2,\ldots,n\}$ written in one-line notation where each entry may have a bar over it (which also corresponds to the group of symmetries of a hypercube, the hyperoctahedral group). It readily follows that there are
$2^{2^{n-1}}(2^{n-1})!$ real-algebra automorphisms of $\mM\mC(n)$, and 
$$
(2^{n-1})!\sum_{k = 0}^{\lfloor 2^{n-2}\rfloor} \frac{2^{2^{n-1}-2k}}{k! (2^{n-1} - 2k)!}
$$
involutions of $\mM\mC(n)$.
The latter is a consequence of known results on signed involutions or wreath products (see \cite{countinginvo} and \cite{countingwr}, respectively), or of a simple counting argument. Indeed, if the underlying permutation has $k$ transpositions and $2^{n-1}-2k$ fixed points, then there are $$\frac{(2^{n-1})!}{2^k k!(2^{n-1}-2k)!}$$ ways to choose this permutation. The signs may be chosen in $2^{2^{n-1}-k}$ ways since each fixed point has an arbitrary sign, and each transposition has two compatible choices of signs.

Finally, if we change the condition $f^2 (\eta) = \eta$ for any $\eta \in \mM \mC (n)$ to the following condition: 
    \begin{equation}
        f^r (\eta ) = \eta \qquad (\forall \eta \in \mM \mC (n)) \label{eq:rInvoCondition}
    \end{equation}
where $r > 2$ is an integer, then similar arguments as above lead to the following formula for the number of automorphisms satisfying \eqref{eq:rInvoCondition}:
$$
        2^{2^{n-1}}\sum_{\sigma\in S_{2^{n-1},r}}\Big(\prod_{k|r,\, r/k \mathrm{\, is\, odd}}\frac{1}{2^{\mathrm{cyc}_k(\sigma)}}\Big).
    $$
In the last formula, the set $S_{2^{n-1}, r}$ is the set of permutations $\sigma$ such that $\sigma^r = \mathrm{Id}$ and $\mathrm{cyc}_k (\sigma )$ is the number of disjoint cycles of length $k$ in $\sigma$. When $r > 2$ is a prime number, the formula becomes
    $$
        (2^{n-1})!\sum_{k=0}^{\lfloor 2^{n-1}/r\rfloor}\frac{2^{(r-1)k}}{k!r^k(2^{n-1}-rk)!} .
    $$


\section{Cauchy--Riemann classes of multicomplex-valued functions}\label{S:classesMonoFunctions}
Inspired by the theory of complex structures (see, for example, \cite{Huybrechts2005}), we associate a complex structure and a Cauchy--Riemann system with every $e$-admissible pair. We compare the resulting classes with multicomplex holomorphy and then study anti-holomorphic and twisted variants.

\subsection{Elliptic-admissible involutions}\label{SS:EAdmissibleCounting}

\begin{definition}\label{D:EAdmissible}
Let $n\geq 1$. An $\mI (n)$-preserving involution $\sigma:\mM \mC(n)\to\mM \mC(n)$ is \emph{elliptic-admissible} (abbreviated $e$-admissible) if there exists $\im{}\in\mI(n)$ such that $\ims{}=-1$ and $\sigma(\im{})=-\im{}$. In that case, $(\sigma,\im{})$ is called an \emph{$e$-admissible pair}.
\end{definition}

The distinguished unit need not be unique. For example, the involution of $\mM \mC(2)$ defined by $\im{1}\mapsto-\im{1}$ and $\im{2}\mapsto-\im{2}$ forms $e$-admissible pairs with both $\im{1}$ and $\im{2}$. By contrast, the swap $\im{1}\mapsto\im{2}$, $\im{2}\mapsto\im{1}$ is not $e$-admissible.

Let $(\sigma,\im{})$ be an $e$-admissible pair, and let $\cL_{\im{}}$ denote multiplication by $\im{}$ on $\mM \mC(n)$. Since $\ims{}=-1$, we have $\cL_{\im{}}^2=-\mathrm{Id}$; hence $\cL_{\im{}}$ is a complex structure on the underlying real vector space. With respect to this structure, $\mM \mC(n)$ has complex dimension $m:=2^{n-1}$ and scalar multiplication
\[
(a+b\sqrt{-1})\cdot\eta:=a\eta+b\im{}\eta,
\qquad a,b\in\mR.
\]
The $e$-admissibility of $(\sigma,\im{})$ means precisely that $\sigma$ is a conjugation for this complex structure.

    \begin{lemma}\label{L:Conjugation}
    Let $\sigma$ be an $\mI (n)$-preserving involution of $\mM \mC(n)$ and let $\im{} \in \mI(n)$ with $\ims{} = -1$. Then $(\sigma , \im{})$ is an $e$-admissible pair if and only if
        \begin{align*}
        \sigma ( \im{} \eta ) = - \im{}\, \sigma (\eta) , \qquad \eta \in \mM \mC(n),
        \end{align*}
    that is, if and only if $\sigma$ is conjugate-linear with respect to the complex structure $\cL_{\im{}}$.
    \end{lemma}

    \begin{proof}
    If $(\sigma, \im{})$ is $e$-admissible, then $\sigma (\im{} \eta ) = \sigma (\im{})\, \sigma (\eta ) = - \im{}\, \sigma (\eta )$ since $\sigma$ is multiplicative. Conversely, taking $\eta = 1$ gives $\sigma (\im{}) = -\im{}$.
    \end{proof}

The next lemma is the analogue of the decomposition $\mC = \mR \oplus \sqrt{-1}\, \mR$ determined by the complex conjugation.

    \begin{lemma}\label{L:FixedDecomposition}
    Let $(\sigma , \im{})$ be an $e$-admissible pair and let
        \begin{align*}
        \cB_{\sigma} := \{ \eta \in \mM \mC(n)\, :\, \sigma (\eta ) = \eta \}
        \end{align*}
    be the fixed set of $\sigma$. Then $\cB_\sigma$ is a real subalgebra of $\mM \mC(n)$ of dimension $2^{n-1}$, and
        \begin{align*}
        \mM \mC(n) = \cB_\sigma \oplus \im{}\, \cB_\sigma.
        \end{align*}
    In other words, every $\eta \in \mM \mC(n)$ can be written uniquely as $\eta = u + \im{} v$ with $u , v \in \cB_\sigma$.
    \end{lemma}

    \begin{proof}
    Since $\sigma$ is a real-algebra automorphism, its fixed set $\cB_\sigma$ is a real subalgebra containing $1$. Every $\eta \in \mM \mC(n)$ decomposes as
    \[
    \eta = \frac{\eta + \sigma (\eta)}{2} + \frac{\eta - \sigma (\eta )}{2},
    \]
    where the first term is fixed by $\sigma$ and the second one is mapped to its negative. It therefore suffices to show that $\{ \eta\, :\, \sigma (\eta ) = - \eta \} = \im{}\, \cB_\sigma$. If $u \in \cB_\sigma$, then $\sigma (\im{} u ) = -\im{}\, \sigma (u) = -\im{} u$ by Lemma~\ref{L:Conjugation}. Conversely, if $\sigma (\zeta ) = - \zeta$, then $\zeta = \im{} ( - \im{} \zeta )$ and $\sigma (- \im{} \zeta) = \im{}\, \sigma (\zeta ) = - \im{} \zeta$, again by Lemma~\ref{L:Conjugation}, so $-\im{} \zeta \in \cB_\sigma$. It remains to prove that the sum is direct. If $u + \im{} v = 0$ with $u , v \in \cB_\sigma$, applying $\sigma$ gives $u - \im{} v = 0$, and adding the two equalities yield $u = 0$, whence $v = - \im{} (\im{} v ) = 0$. Finally, since $\cL_{\im{}}$ is invertible, $\dim_{\mR} \im{} \cB_\sigma = \dim_{\mR} \cB_\sigma$, and the equality $2 \dim_{\mR} \cB_\sigma = 2^n$ gives the dimension claim.
    \end{proof}

\subsection{General Cauchy--Riemann equations}\label{SS:GeneralCR}
The subalgebra $\cB_\sigma$ is thus a \emph{real form} of the complex vector space $(\mM \mC(n) , \cL_{\im{}})$, exactly as $\mR$ is a real form of $\mC$ with respect to the usual conjugation. This suggests the following definition.

    \begin{definition}\label{D:SigmaCR}
    Let $(\sigma , \im{})$ be an $e$-admissible pair, let $U \subseteq \mM \mC(n)$ be open, and let $F : U \ra \mM \mC(n)$ be of class $C^1$. We say that $F$ satisfies the \emph{$(\sigma , \im{})$-Cauchy--Riemann equations} on $U$ if
        \begin{equation}\label{Eq:SigmaCR}
        D F (\eta ) [ \im{} \beta ] = \im{}\, DF (\eta ) [\beta] , \qquad \eta \in U , \ \beta \in \cB_\sigma.
        \end{equation}
    \end{definition}

To justify the terminology, choose a real basis $\beta_0 = 1 , \beta_1 , \ldots , \beta_{m - 1}$ of $\cB_\sigma$, where $m = 2^{n-1}$. By Lemma~\ref{L:FixedDecomposition}, every $\eta \in \mM \mC(n)$ can be written uniquely as
    \begin{align*}
    \eta = \sum_{\ell = 0}^{m-1} ( u_\ell + \im{} v_\ell )\, \beta_\ell , \qquad u_\ell , v_\ell \in \mR,
    \end{align*}
and the functions $(u_0, v_0 , \ldots , u_{m-1}, v_{m-1})$ form a system of real coordinates on $\mM \mC(n)$. Since $\partial F / \partial u_\ell = DF (\eta ) [\beta_\ell]$ and $\partial F / \partial v_\ell = DF (\eta ) [\im{} \beta_\ell]$, condition \eqref{Eq:SigmaCR} is equivalent to the system of $m$ equations
    \begin{equation}\label{Eq:SigmaCRCoordinates}
    \frac{\partial F}{\partial v_\ell} = \im{} \frac{\partial F}{\partial u_\ell} , \qquad 0 \leq \ell \leq m - 1,
    \end{equation}
or, introducing the Wirtinger-type operators
    \begin{align*}
    \overline{\partial}_\ell^{\, \sigma , \im{}} := \frac{1}{2} \Big( \frac{\partial}{\partial u_\ell} + \im{} \frac{\partial}{\partial v_\ell} \Big) , \qquad 0 \leq \ell \leq m -1,
    \end{align*}
to the system $\overline{\partial}_\ell^{\, \sigma , \im{}} F = 0$, $0 \leq \ell \leq m - 1$. Moreover, writing $F = P + \im{} Q$, where $P$ and $Q$ take values in $\cB_\sigma$, and comparing the components in the decomposition of Lemma~\ref{L:FixedDecomposition}, the system \eqref{Eq:SigmaCRCoordinates} takes the classical form
    \begin{align*}
    \frac{\partial P}{\partial u_\ell} = \frac{\partial Q}{\partial v_\ell} , \qquad \frac{\partial P}{\partial v_\ell} = - \frac{\partial Q}{\partial u_\ell}, \qquad 0 \leq \ell \leq m - 1.
    \end{align*}

The following theorem identifies exactly the class of functions defined by the $(\sigma , \im{})$-Cauchy--Riemann equations. It also reveals that the class depends only on the unit $\im{}$, and not on the involution $\sigma$. In the following statement, the imaginary number $i$ (not in bold) refers to the classical imaginary unit of the set of complex numbers $\mC$. 

    \begin{theorem}\label{T:SigmaIndependence}
    Let $(\sigma , \im{})$ be an $e$-admissible pair, let $U \subseteq \mM \mC(n)$ be open, and let $F : U \ra \mM \mC(n)$ be of class $C^1$. The following statements are equivalent.
        \begin{enumerate}[label=(\roman*)]
        \item $F$ satisfies the $(\sigma , \im{})$-Cauchy--Riemann equations on $U$;
        \item $DF (\eta ) \circ \cL_{\im{}} = \cL_{\im{}} \circ D F (\eta)$ for every $\eta \in U$;
        \item for one (equivalently, for every) real basis $\beta_0 , \ldots , \beta_{m-1}$ of $\cB_\sigma$, the function $F$, viewed in the complex coordinates $z_\ell := u_\ell + i\, v_\ell$, $0 \leq \ell \leq m - 1$ where $i^2 = -1$, is a holomorphic mapping of $m$ complex variables.
        \end{enumerate}
    In particular, the set of solutions of the $(\sigma , \im{})$-Cauchy--Riemann equations depends only on the unit $\im{}$, and not on the involution $\sigma$.
    \end{theorem}

    \begin{proof}
    (ii) $\Ra$ (i) is immediate since $\cB_\sigma \subseteq \mM \mC(n)$.

    (i) $\Ra$ (ii). Let $h \in \mM \mC(n)$ and write $h = u + \im{} v$ with $u , v \in \cB_\sigma$, using Lemma~\ref{L:FixedDecomposition}. Then $\im{} h = \im{} u - v$ and, using \eqref{Eq:SigmaCR} twice together with the real-linearity of $DF (\eta )$,
    \begin{align*}
    DF (\eta ) [\im{} h] &= DF (\eta )[\im{} u] - D F (\eta ) [v] = \im{}\, DF (\eta )[u] - DF (\eta ) [v] , \\
    \im{}\, DF (\eta ) [h] &= \im{}\, DF (\eta) [u] + \im{}\, DF (\eta ) [\im{} v] = \im{}\, DF (\eta )[u] + \ims{}\, DF (\eta )[v].
    \end{align*}
    The two right-hand sides agree because $\ims{} = -1$.

    (ii) $\Leftrightarrow$ (iii). Fix a real basis $\beta_0 , \ldots , \beta_{m-1}$ of $\cB_\sigma$ and let $\Phi : \mM \mC(n) \ra \mC^m$ be the real-linear isomorphism sending $\eta = \sum_\ell (u_\ell + \im{} v_\ell )\, \beta_\ell$ to $(u_0 + i\, v_0 , \ldots , u_{m-1} + i\, v_{m-1})$. By construction, $\Phi \circ \cL_{\im{}} = ( i \, \cdot ) \circ \Phi$, where $i \, \cdot$ denotes the componentwise multiplication by $i$ on $\mC^m$. The function $G := \Phi \circ F \circ \Phi^{-1}$, defined on the open set $\Phi (U ) \subseteq \mC^m$, is of class $C^1$ and its differential is $DG = \Phi \circ DF \circ \Phi^{-1}$ at corresponding points. Condition (ii) is therefore equivalent to the complex-linearity of the differential of $G$ at every point of $\Phi (U)$, which is the definition of a holomorphic mapping in $m$ complex variables (see, e.g., \cite[Chapter 1]{Krantz2001}). This is statement (iii). Since condition (ii) does not refer to the basis, statement (iii) holds for one basis if and only if it holds for every basis.

    The last claim of the theorem is clear, since condition (ii) involves only $\im{}$.
    \end{proof}

In view of Theorem~\ref{T:SigmaIndependence}, for an open set $U \subseteq \mM \mC(n)$ and a unit $\im{} \in \mI(n)$ with $\ims{} = -1$, we denote by $\hol_{\im{}} (U)$ the set of functions of class $C^1$ satisfying condition (ii) of the theorem, and we call its elements \emph{$\im{}$-holomorphic functions}. Note that condition (ii) makes sense for every such unit $\im{}$, whether or not an involution is specified. 

Theorem~\ref{T:SigmaIndependence} has several immediate consequences. Since holomorphic mappings of several complex variables are real-analytic, every $\im{}$-holomorphic function is real-analytic. 
Moreover, a computation based on condition (ii) and the Leibniz rule shows that $\hol_{\im{}} (U)$ is a real algebra containing the constants. Second, and more importantly for our purposes, the involution $\sigma$ itself does not affect the class of functions obtained: two $e$-admissible pairs $(\sigma_1 , \im{})$ and $(\sigma_2 , \im{})$ sharing the same unit determine the same class of $\im{}$-holomorphic functions, presented in different real coordinate systems. The involution selects the real form $\cB_\sigma$ of the complex vector space $(\mM \mC(n) , \cL_{\im{}})$, but the underlying function theory is governed by $\im{}$ alone as a consequence of (ii). 

We now compare the classes $\hol_{\im{}} (U)$ with the multicomplex holomorphic functions. Recall that an odd vector $\vec{b} \in \mathbb{F}_2^n$ is characterized by the condition $\vec{e}^\top \vec{b} \equiv 1 \pmod 2$. 

    \begin{proposition}\label{P:HoloInclusion}
    Let $U \subseteq \mM \mC(n)$ be a nonempty open set and let $\im{}\in\mI(n)$ with $\ims{}=-1$. Then $\hol(U)\subseteq\hol_{\im{}}(U)$, and the inclusion is strict when $n\geq2$.
    \end{proposition}

    \begin{proof}
    If $F \in \hol(U)$ with $DF (\eta )[h] = \lambda (\eta )\, h$, then, by the commutativity of multiplication,
    \[
    DF (\eta) [\im{} h] = \lambda (\eta )\, \im{} h = \im{}\, \lambda (\eta )\, h = \im{}\, DF (\eta )[h],
    \]
    so condition (ii) of Theorem~\ref{T:SigmaIndependence} holds and $F \in \hol_{\im{}} (U)$.

    For the strictness when $n \geq 2$, assume first that $\im{} = \im{1}$ and consider the real-linear function $F (\eta ) := x_{\vec 0} + x_{\vec e_1} \im{1}$, for which $DF (\eta ) = F$ for every $\eta$. Fix $\eta, h \in \mM \mC(n)$. Comparing the coefficients of $1$ and $\im{1}$ in $\im{1} h$ and in $h$, we find 
        \[
        DF(\eta) [\im{1} h] = F(\im{1}h) = -x_{\vec{e}_1} + x_{\vec{0}}\, \im{1} = \im{1} F(h) = \im{1} DF(\eta)[h],
        \]
    so $F$ is $\im{1}$-holomorphic. However, $F (\im{2}) = 0$ while $\im{2} F (1) = \im{2} \neq 0$, so $DF$ is not $\mM \mC(n)$-linear and $F \notin \hol(U)$. For an arbitrary direction $\im{} = \mathbf i^{\vec b}$, complete the odd vector $\vec b$ to a basis $\vec t_1 = \vec b , \vec t_2 , \ldots , \vec t_n$ of $\mathbb F_2^n$ consisting of odd vectors (if $j_0 \in \operatorname{supp} (\vec b)$, the standard basis vectors $\vec e_j$, $j \neq j_0$, complete $\vec b$ to such a basis).
    
      The assignment $\psi (\im{j}) := \mathbf i^{\vec t_j}$, $1 \leq j \leq n$, extends to a real-algebra endomorphism of $\mM \mC(n)$ by the discussion preceding Lemma~\ref{L:SignVectors}, and $\psi$ is bijective because the matrix $T$ with columns $\vec t_1 , \ldots , \vec t_n$ is invertible over $\mathbb F_2$.
    
      Since $D (\psi \circ F \circ \psi^{-1}) = \psi \circ DF \circ \psi^{-1}$ on $\psi (U)$ and $\psi^{-1} \circ \cL_{\im{}} \circ \psi = \cL_{\psi^{-1} (\im{})} = \cL_{\im{1}}$ on $\mM \mC (n)$, the map $F \mapsto \psi \circ F \circ \psi^{-1}$ sends $\hol_{\im{1}} (\psi^{-1} (U))$ onto $\hol_{\im{}} (U)$ and preserves multicomplex holomorphy in both directions, so the strictness transfers from $\im{1}$ to $\im{}$.
    \end{proof}

    \begin{remarks}
    (1) One could try to attach a different system of equations to an $\mI(n)$-preserving involution $\sigma$ by using $\sigma$ as a change of generators rather than as a conjugation: the units $J_k := \sigma (\im{k})$, $1 \leq k \leq n$, again commute, square to $-1$, and generate $\mM \mC(n)$, and one may write the analogue of the system \eqref{Eq:MCCRSystem} with respect to them. However, as observed after Proposition~\ref{P:MCHoloSystem}, the resulting system characterizes multicomplex holomorphy again, and no new classes of functions arise in this way. This is why the conjugation viewpoint of Lemma~\ref{L:Conjugation}, which requires the $e$-admissibility of the involution, is the appropriate one.

    (2) The definition of an $e$-admissible pair requires $\ims{} = -1$. If instead $\jm{} \in \mI(n)$, with $n \geq 2$, satisfies $\jm{}^2 = 1$, $\jm{} \neq \pm 1$, and $\sigma (\jm{} ) = - \jm{}$, then the involution is called \emph{hyperbolic-admissible} (abbreviated $h$-admissible). In this context, the proof of Lemma~\ref{L:FixedDecomposition} may be adapted and yields the decomposition $\mM \mC(n) = \cB_\sigma \oplus \jm{}\, \cB_\sigma$, and the analogue of \eqref{Eq:SigmaCR} produces first-order systems of \emph{hyperbolic} type, since $\cL_{\jm{}}^2 = + \mathrm{Id}$. See \cite{Catoni2008} for more on the subject of hyperbolic Cauchy--Riemann equations and on general commutative hypercomplex numbers and their connections to the geometry of Minkowski space-time.    \end{remarks}

The next theorem shows that to describe multicomplex holomorphy, we may use $\im{}$-holomorphy for a subclass of imaginary units $\im{}$, provided their sign vectors $\vec{b}$ span $\mathbb{F}_2^n$. Recall that an odd vector $\vec{b} \in \mathbb{F}_2^n$ is characterized by the condition $\vec{e}^\top \vec{b} \equiv 1 \pmod 2$. 

\begin{theorem}\label{T:EquivalenceSpanSWithHolomorphy}
    Let $U \subseteq \mM \mC (n)$ be open with $U \neq \varnothing$ and let $S \subseteq \mathbb{F}_2^n$ be a subset of odd vectors. Then the following statements are equivalent. 
    \begin{enumerate}[label=(\roman*)]
        \item $\displaystyle\hol (U) = \bigcap_{\vec{s} \in S} \hol_{\im{}^{\vec{s}}} (U)$.
        \item $\mathrm{span} \; S = \mathbb{F}_2^n$. 
    \end{enumerate}
\end{theorem}
\begin{proof}
    We first show that (i) implies (ii) by contrapositive. Let $\mathrm{span} \, S \neq \mathbb{F}_2^n$. Then, there is a $k \in \{ 1, 2, \ldots , n \}$ such that $\vec{e}_k \not\in \mathrm{span} \; S$. We will construct a function $F$ such that $F \in \cap_{\vec{s} \in S} \hol_{\im{}^{\vec{s}}} (U)$, but $F \not\in \hol (U)$. 
    
    Consider $\{\vec{s}_1, \vec{s}_2, \ldots , \vec{s}_m \}$ a basis for $\mathrm{span} \, S$ and complete the set $\{ \vec{s}_1 , \vec{s}_2 , \ldots , \vec{s}_m, \vec{e}_k \}$ as a basis $B$ for $\mathbb{F}_2^n$. Let $H$ be the spanning set of the element in $B \backslash \{ \vec{e}_k \}$. Now set $\mathcal{H} := \mathrm{span}_{\mathbb{R}} \; \{ \im{}^{\vec{c}} \, : \, \vec{c} \in H \}$. Since $\im{}^{\vec{e}_k} = \im{k}$, it is easily seen that 
    $$ 
    \im{k} \mathcal{H} = \operatorname{span} \{ \im{}^{\vec{c}} \, : \, \vec{c} \not\in H \}
    $$
    and that $\mM \mC (n) = \mathcal{H} \oplus \im{k} \mathcal{H}$. Define $F : U \ra \mM \mC (n)$ as the projection onto $\mathcal{H}$ along $\im{k} \mathcal{H}$, that is
        $$  
            F (\eta) = F (u + \im{k} v) := u,
        $$
    where $u, v \in \mathcal{H}$. Then $F$ is real-differentiable with $DF (\eta) = F$, for any $\eta \in U$. For any $\vec{s} \in S$, it is $\im{}^{\vec{s}}$-holomorphic on $U$. Indeed, for $\vec{s} \in S$, for $\eta \in U$, and for $h = u + \im{k} v$ with $u, v \in \mathcal{H}$, we have
        $$
            \im{}^{\vec{s}} DF (\eta) [h] = \im{}^{\vec{s}} F [h] = \im{}^{\vec{s}} u = F [\im{}^{\vec{s}} h] = DF (\eta ) [\im{}^{\vec{s}} h].
        $$
    However, for $h = \im{k} v$ with $0\neq v \in \mathcal{H}$, we have $DF (\eta ) [h] = F [\im{k} v] = 0$, but $\im{k} DF(\eta) [v] = \im{k} v$, hence $F$ is not holomorphic. Condition (i) does not hold. 

    We now prove (ii) implies (i). Suppose that $\mathrm{span} \, S = \mathbb{F}_2^n$. From Proposition \ref{P:HoloInclusion}, $\hol (U) \subseteq \cap_{\vec{s} \in S} \hol_{\im{}^{\vec{s}}} (U)$. We prove the reverse inclusion. Let $F \in \cap_{\vec{s} \in S} \hol_{\im{}^{\vec{s}}} (U)$. Fix $k \in \{ 1, 2, \ldots , n \}$. Since $\vec{e}_k \in \mathrm{span} \; S$, we can write $\vec{e}_k = \vec{s}_1 \oplus \vec{s}_2 \oplus \cdots \oplus \vec{s}_m$, with $m \leq n$. From the rule \eqref{Eq:MonomialMult} for multiplying monomials, we immediately see that
        $$
            \mathcal{L}_{\im{}^{\vec{e}_k}} \circ DF (\eta ) = DF (\eta ) \circ \mathcal{L}_{\im{}^{\vec{e}_k}} .
        $$
    Since $k$ was arbitrary, the last commuting property is true for every standard basis vector of $\mathbb{F}_2^n$. Since all the generators commute, one has
$\cL_{\im{}^{\vec b}} = \cL_{\im{1}}^{\,b_1}\cdots\cL_{\im{n}}^{\,b_n}$ for every $\vec b \in \mathbb{F}_2^{n}$. Hence $DF(\eta)$ commutes with $\cL_{\im{}^{\vec b}}$ for all $\vec b$, and by $\mathbb{R}$-linearity with $\cL_h$, the multiplication operator by $h$, for every $h \in \mM \mC(n)$. Applying this to the multicomplex number $1$ gives 
    \[
    DF(\eta)[h] = DF(\eta)\bigl[\cL_h(1)\bigr]
    = \cL_h\bigl(DF(\eta)[1]\bigr) = h\,\lambda(\eta),
    \]
with $\lambda (\eta ) := DF (\eta ) [1]$. So $DF(\eta)$ is multiplication by $\lambda(\eta)$ and therefore $\mM \mC(n)$-linear.
    Thus $F \in \hol(U)$.
\end{proof}

Since $\{ \vec{e}_1, \vec{e}_2, \ldots , \vec{e}_n \}$ is a basis for $\mathbb{F}_2^n$, we obtain the following corollary, stating that the generators are sufficient to characterize multicomplex holomorphy in terms of $\im{}$-holomorphy.

\begin{corollary}\label{C:hol-intersection}
Let $U \subseteq \mM \mC(n)$ be open. A function $F \colon U \to \mM \mC(n)$ of class $C^1$
is multicomplex holomorphic if and only if it is $\im{k}$-holomorphic for each of the
generators $\im{1}, \dots, \im{n}$. Equivalently,
\[
\hol(U) \;=\; \bigcap_{k=1}^{n} \hol_{\im{k}}(U).
\]
\end{corollary}

\begin{remark}
The intersection ranges over the $n$ generators only, although $\mI (n)$ contains
$2^{n-1}$ units squaring to $-1$, namely the monomials $\im{}^{\vec b}$ with
$\vec e^{\top}\vec b \equiv 1 \pmod 2$. The other $2^{n-1}-n$ of them are redundant: for any such
unit, $\cL_{\im{}^{\vec b}} = \cL_{\im{1}}^{\,b_1}\cdots\cL_{\im{n}}^{\,b_n}$
commutes with $DF(\eta)$ as soon as the generators do, so
$\bigcap_{k=1}^{n}\hol_{\im{k}}(U) \subseteq \hol_{\im{}^{\vec b}}(U)$. 
\end{remark}

\subsection{The bicomplex case}\label{SS:Bicomplex}

We now describe the situation completely for $n = 2$. The set of $\mI(2)$-preserving involutions of $\mM \mC(2)$ consists of the identity, the three conjugations\footnote{In \cite{Bicomplex}, the principal conjugates are denoted respectively by $\overline{\phantom{z}}$, $\dagger$, and $\ast$.}
    \begin{align*}
    \sigma_1 : \ \im{1} \mapsto - \im{1} , \ \im{2} \mapsto \im{2} , \qquad
    \sigma_2 : \ \im{1} \mapsto \im{1} , \ \im{2} \mapsto - \im{2} , \qquad
    \sigma_3 : \ \im{1} \mapsto - \im{1} , \ \im{2} \mapsto - \im{2},
    \end{align*}
and the two swaps\footnote{The unsigned swap appears, for example, in \cite{PerezRegaladoQuiroga2018}.}
    \begin{align*}
    \omega : \ \im{1} \mapsto \im{2} , \ \im{2} \mapsto \im{1} , \qquad \widetilde{\omega} : \ \im{1} \mapsto -\im{2} , \ \im{2} \mapsto -\im{1}.
    \end{align*}
The admissible directions are $\im{1}$ and $\im{2}$, since the remaining nontrivial unit $\im{1}\im{2}$ squares to $+1$. Among the six involutions, exactly three are $e$-admissible, namely $\sigma_1, \sigma_2, \sigma_3$, and they form exactly four $e$-admissible pairs,
    \begin{align*}
    (\sigma_1 , \im{1}) , \qquad (\sigma_2 , \im{2}) , \qquad (\sigma_3 , \im{1}) , \qquad (\sigma_3 , \im{2}).
    \end{align*}

The fixed algebras of the three conjugations are
    \begin{align*}
    \cB_{\sigma_1} = \spn_{\mR} \{ 1 , \im{2} \}, \qquad \cB_{\sigma_2} = \spn_{\mR} \{ 1 , \im{1} \}, \qquad \cB_{\sigma_3} = \spn_{\mR} \{ 1 , \im{1}\im{2} \}.
    \end{align*}
Write $\eta = x_0 + x_1 \im{1} + x_2 \im{2} + x_{12}\, \im{1}\im{2}$ for the canonical coordinates on $\mM \mC(2)$. For the pair $(\sigma_1 , \im{1})$, the basis $\beta_0 = 1$, $\beta_1 = \im{2}$ of $\cB_{\sigma_1}$ gives $u_0 = x_0$, $v_0 = x_1$, $u_1 = x_2$, $v_1 = x_{12}$, and the system \eqref{Eq:SigmaCRCoordinates} reads
    \begin{equation}\label{Eq:BicomplexSystem1}
    \frac{\partial F}{\partial x_1} = \im{1} \frac{\partial F}{\partial x_0} , \qquad \frac{\partial F}{\partial x_{12}} = \im{1} \frac{\partial F}{\partial x_2}.
    \end{equation}
For the pair $(\sigma_3 , \im{1})$, the basis $\beta_0 = 1$, $\beta_1 = \im{1}\im{2}$ of $\cB_{\sigma_3}$ gives $u_0 = x_0$, $v_0 = x_1$, $u_1 = x_{12}$, $v_1 = - x_2$, and the system \eqref{Eq:SigmaCRCoordinates} reads
    \begin{equation}\label{Eq:BicomplexSystem3}
    \frac{\partial F}{\partial x_1} = \im{1} \frac{\partial F}{\partial x_0} , \qquad \frac{\partial F}{\partial x_{2}} = -\im{1} \frac{\partial F}{\partial x_{12}}.
    \end{equation}
Although the two systems look different, multiplying the second equation of \eqref{Eq:BicomplexSystem1} by $\im{1}$ shows that they have exactly the same solutions, as predicted by Theorem~\ref{T:SigmaIndependence}: both characterize the class $\hol_{\im{1}} (U)$, which consists of the holomorphic functions of the two complex variables $z_0 = x_0 + \im{1} x_1$ and $z_1 = x_2 + \im{1} x_{12}$, with values in the complex vector space $(\mM \mC(2) , \cL_{\im{1}})$. The two pairs directed by $\im{2}$ lead similarly, in the complex variables $w_0 = x_0 + \im{2} x_2$ and $w_1 = x_1 + \im{2} x_{12}$, to the class $\hol_{\im{2}} (U)$.

    \begin{theorem}\label{T:BicomplexComplete}
    Let $U \subseteq \mM \mC(2)$ be a nonempty open set and let $F : U \ra \mM \mC(2)$ be of class $C^1$. Then the following statements hold.
        \begin{enumerate}[label=(\alph*)]
        \item The four $e$-admissible pairs of $\mM \mC(2)$ give rise to exactly two classes of functions, namely $\hol_{\im{1}} (U)$ and $\hol_{\im{2}} (U)$.
        \item $F$ is bicomplex holomorphic if and only if $F \in \hol_{\im{1}} (U) \cap \hol_{\im{2}} (U)$.
        \item The inclusions $\hol(U) \subseteq \hol_{\im{1}} (U)$ and $\hol(U) \subseteq \hol_{\im{2}} (U)$ are strict, and $\hol_{\im{1}} (U) \neq \hol_{\im{2}} (U)$.
        \end{enumerate}
    \end{theorem}

    \begin{proof}
    Part (a) follows from Theorem~\ref{T:SigmaIndependence} and the list of pairs above.

    (b) If $F$ is bicomplex holomorphic, then $F \in \hol_{\im{1}} (U) \cap \hol_{\im{2}} (U)$ by Proposition~\ref{P:HoloInclusion}. Conversely, suppose that $DF (\eta )$ commutes with both $\cL_{\im{1}}$ and $\cL_{\im{2}}$. Then $DF (\eta )$ also commutes with $\cL_{\im{1}} \circ \cL_{\im{2}} = \cL_{\im{1}\im{2}}$, and hence, for $h = c_0 + c_1 \im{1} + c_2 \im{2} + c_{12}\, \im{1}\im{2}$ with real coefficients,
    \begin{align*}
    DF (\eta ) [h] &= c_0\, DF(\eta)[1] + c_1 \cL_{\im{1}} DF(\eta)[1] + c_2 \cL_{\im{2}} DF(\eta)[1] + c_{12} \cL_{\im{1}\im{2}} DF(\eta)[1] \\
    &= h\, DF(\eta)[1].
    \end{align*}
    Thus $DF (\eta )$ is $\mM \mC(2)$-linear and $F$ is bicomplex holomorphic.

    (c) Consider the real-linear function $F (\eta ) := x_0 + x_1 \im{1}$, for which $DF (\eta ) = F$ for every $\eta$. On the one hand, $F (\im{1} \eta ) = - x_1 + x_0\, \im{1} = \im{1} F (\eta )$, so $F \in \hol_{\im{1}} (U)$. On the other hand, $F (\im{2} \cdot 1 ) = 0$ while $\im{2}\, F (1) = \im{2} \neq 0$, so $F \notin \hol_{\im{2}} (U)$, and in particular $F$ is not bicomplex holomorphic. Exchanging the roles of $\im{1}$ and $\im{2}$ completes the proof.
    \end{proof}

In the bicomplex setting, every Cauchy--Riemann system generated by an $e$-admissible involution of $\mM \mC(2)$ characterizes either the $\im{1}$-holomorphic or the $\im{2}$-holomorphic functions, and bicomplex holomorphy is exactly the conjunction of the two.

\subsection{A tricomplex example}\label{SS:Tricomplex}

We conclude with an example in $\mM \mC(3)$ showing that, in contrast with the bicomplex case, the real form $\cB_\sigma$ selected by an $e$-admissible involution need not be spanned by elements of $\pm \mI(n)$. Consider the involution $\sigma$ of $\mM \mC(3)$ determined by
    \begin{align*}
    \sigma (\im{1}) = - \im{1} , \qquad \sigma (\im{2}) = -\im{2} , \qquad \sigma (\im{3}) = \im{1}\im{2}\im{3} .
    \end{align*}
Its exponent matrix $M$ has columns $\vec e_1$, $\vec e_2$, $(1,1,1)^\top$, and $\ker (M - I) = \{ \vec b \in \mathbb F_2^3\, :\, b_3 = 0 \}$, whose odd vectors are $\vec e_1$ and $\vec e_2$. Since $\sigma (\im{1}) = -\im{1}$ and $\sigma (\im{2}) = - \im{2}$, both $(\sigma , \im{1})$ and $(\sigma , \im{2})$ are $e$-admissible pairs. We work with $(\sigma , \im{1})$.

A direct computation shows that the four elements
    \begin{align*}
    \beta_0 = 1 , \qquad \beta_1 = \im{1}\im{2} , \qquad \beta_2 = \im{3} + \im{1}\im{2}\im{3} , \qquad \beta_3 = \im{1}\im{3} + \im{2}\im{3}
    \end{align*}
are fixed by $\sigma$: for instance, $\sigma (\im{1}\im{3}) = (-\im{1})(\im{1}\im{2}\im{3}) = \im{2}\im{3}$ and $\sigma (\im{2}\im{3}) = (-\im{2})(\im{1}\im{2}\im{3}) = \im{1}\im{3}$, so $\beta_3$ is fixed. They are linearly independent over $\mR$ and, since $\dim_{\mR} \cB_\sigma = 2^{3-1} = 4$ by Lemma~\ref{L:FixedDecomposition}, they form a basis of $\cB_\sigma$. In particular, $\cB_\sigma$ is \emph{not} spanned by signed unit monomials: the involution $\sigma$ glues the units $\im{3}$ and $\im{1}\im{2}\im{3}$, as well as the units $\im{1}\im{3}$ and $\im{2}\im{3}$. This phenomenon cannot occur for $n \leq 2$, as the fixed algebras displayed in Section~\ref{SS:Bicomplex} show.

Write
    \begin{align*}
    \eta = x_0 + x_1 \im{1} + x_2 \im{2} + x_3 \im{3} + x_{12}\, \im{1}\im{2} + x_{13}\, \im{1}\im{3} + x_{23}\, \im{2}\im{3} + x_{123}\, \im{1}\im{2}\im{3}
    \end{align*}
for the canonical coordinates on $\mM \mC(3)$. Expanding $\eta = \sum_{\ell = 0}^{3} (u_\ell + \im{1} v_\ell )\, \beta_\ell$ and comparing the coefficients yields the dictionary
    \begin{align*}
    &u_0 = x_0 , \qquad v_0 = x_1 , \qquad u_1 = x_{12} , \qquad v_1 = - x_2 , \\
    &u_2 = \frac{x_3 + x_{123}}{2} , \qquad v_2 = \frac{x_{13} - x_{23}}{2} , \qquad u_3 = \frac{x_{13} + x_{23}}{2} , \qquad v_3 = \frac{x_{123} - x_3}{2},
    \end{align*}
and the $(\sigma , \im{1})$-Cauchy--Riemann equations $\partial F / \partial v_\ell = \im{1}\, \partial F / \partial u_\ell$, $0 \leq \ell \leq 3$, become, in the canonical coordinates,
    \begin{equation}\label{Eq:TricomplexSystem}
    \begin{aligned}
    \frac{\partial F}{\partial x_1} &= \im{1} \frac{\partial F}{\partial x_0} , &
    - \frac{\partial F}{\partial x_2} &= \im{1} \frac{\partial F}{\partial x_{12}} , \\
    \frac{\partial F}{\partial x_{13}} - \frac{\partial F}{\partial x_{23}} &= \im{1} \Big( \frac{\partial F}{\partial x_3} + \frac{\partial F}{\partial x_{123}} \Big) , &
    - \frac{\partial F}{\partial x_3} + \frac{\partial F}{\partial x_{123}} &= \im{1} \Big( \frac{\partial F}{\partial x_{13}} + \frac{\partial F}{\partial x_{23}} \Big).
    \end{aligned}
    \end{equation}
The last two equations genuinely mix the four coordinates $x_3 , x_{13}, x_{23}, x_{123}$, reflecting the gluing performed by $\sigma$.

For comparison, the diagonal conjugation $\tau : \im{k} \mapsto - \im{k}$, $1 \leq k \leq 3$, also forms an $e$-admissible pair with $\im{1}$, and its fixed algebra $\cB_{\tau} = \spn_{\mR} \{ 1 , \im{1}\im{2} , \im{1}\im{3} , \im{2}\im{3} \}$ is spanned by unit monomials. The corresponding system is
    \begin{align*}
    \frac{\partial F}{\partial x_1} = \im{1} \frac{\partial F}{\partial x_0} , \qquad
    - \frac{\partial F}{\partial x_2} = \im{1} \frac{\partial F}{\partial x_{12}} , \qquad
    - \frac{\partial F}{\partial x_3} = \im{1} \frac{\partial F}{\partial x_{13}} , \qquad
    \frac{\partial F}{\partial x_{123}} = \im{1} \frac{\partial F}{\partial x_{23}}.
    \end{align*}
By Theorem~\ref{T:SigmaIndependence}, this system and the system \eqref{Eq:TricomplexSystem} have exactly the same solutions, namely the class $\hol_{\im{1}} (U)$ of holomorphic functions of four complex variables with respect to the complex structure $\cL_{\im{1}}$. One can also verify the equivalence of the two systems directly by substitution. The exotic involution $\sigma$ does not change the function theory, but it changes the distinguished real form and the resulting real presentation of the Cauchy--Riemann equations.

\subsection{Anti-holomorphic and twisted classes}\label{SS:Beyond}

We present two families of functions connected to the classes $\hol_{\im{}} (U)$. The first family consists of the \emph{anti-holomorphic} classes, obtained by reversing the sign in the Cauchy--Riemann equations \eqref{Eq:SigmaCR}. They turn out to be conjugates of the classes $\hol_{\im{}} (U)$, and it is here that the involutions finally play an operational role as the maps implementing the conjugation. The second family consists of the \emph{twisted} classes, in which the differential intertwines two different complex structures. We show that every such class is an automorphic image of an $\im{}$-holomorphy class.
We begin with the anti-holomorphic classes.

    \begin{definition}\label{D:AntiCR}
    Let $(\sigma,\im{})$ be an $e$-admissible pair, let $U\subseteq\mM \mC(n)$ be open, and let $F:U\to\mM \mC(n)$ be of class $C^1$. We say that $F$ satisfies the \emph{$(\sigma , \im{})$-anti-Cauchy--Riemann equations} on $U$ if
        \begin{equation}\label{Eq:AntiCR}
        D F (\eta ) [ \im{} \beta ] = - \im{}\, DF (\eta ) [\beta] , \qquad \eta \in U , \ \beta \in \cB_\sigma.
        \end{equation}
    \end{definition}

In the coordinates of Section~\ref{SS:GeneralCR}, condition \eqref{Eq:AntiCR} reads $\partial F / \partial v_\ell = - \im{}\, \partial F / \partial u_\ell$ for $0 \leq \ell \leq m - 1$, that is, $\partial_\ell^{\, \sigma , \im{}} F = 0$ for the unbarred Wirtinger-type operators
    \begin{align*}
    \partial_\ell^{\, \sigma , \im{}} := \frac{1}{2} \Big( \frac{\partial}{\partial u_\ell} - \im{} \frac{\partial}{\partial v_\ell} \Big) , \qquad 0 \leq \ell \leq m - 1.
    \end{align*}

    \begin{proposition}\label{P:AntiHolo}
    Let $(\sigma,\im{})$ be an $e$-admissible pair, let $U\subseteq\mM \mC(n)$ be open, and let $F:U\to\mM \mC(n)$ be of class $C^1$. The following statements are equivalent.
        \begin{enumerate}[label=(\roman*)]
        \item $F$ satisfies the $(\sigma , \im{})$-anti-Cauchy--Riemann equations on $U$;
        \item $DF (\eta ) \circ \cL_{\im{}} = - \cL_{\im{}} \circ DF (\eta )$ for every $\eta \in U$;
        \item $\sigma \circ F$ satisfies the $(\sigma , \im{})$-Cauchy--Riemann equations on $U$;
        \item $F \circ \sigma$ satisfies the $(\sigma , \im{})$-Cauchy--Riemann equations on $\sigma (U)$.
        \end{enumerate}
    In particular, the class of solutions, denoted by $\overline{\hol}_{\im{}} (U)$, depends only on the unit $\im{}$, and
        \begin{align*}
        \overline{\hol}_{\im{}} (U) = \{ \sigma \circ G\, :\, G \in \hol_{\im{}} (U) \} = \{ G \circ \sigma\, :\, G \in \hol_{\im{}} (\sigma (U)) \}
        \end{align*}
    for every involution $\sigma$ such that $(\sigma,\im{})$ is an $e$-admissible pair.
    \end{proposition}

    \begin{proof}
    (i) $\Leftrightarrow$ (ii). As in the proof of Theorem~\ref{T:SigmaIndependence}, write $h = u + \im{} v$ with $u , v \in \cB_\sigma$, so that $\im{} h = \im{} u - v$. If (i) holds, then
    \begin{align*}
    DF (\eta ) [\im{} h] &= DF (\eta )[\im{} u] - D F (\eta ) [v] = -\im{}\, DF (\eta )[u] - DF (\eta ) [v] , \\
    -\im{}\, DF (\eta ) [h] &= -\im{}\, DF (\eta) [u] - \im{}\, DF (\eta ) [\im{} v] = -\im{}\, DF (\eta )[u] - DF (\eta )[v],
    \end{align*}
    and the two right-hand sides agree, which gives (ii). The converse is immediate.

    (ii) $\Leftrightarrow$ (iii). Since $\sigma$ is linear, $D (\sigma \circ F) (\eta ) = \sigma \circ DF (\eta )$, and Lemma~\ref{L:Conjugation} gives $\sigma \circ \cL_{\im{}} = - \cL_{\im{}} \circ \sigma$. Hence, since $DF (\eta )$ is real-linear,
    \[
    D (\sigma \circ F)(\eta ) \circ \cL_{\im{}} = \sigma \circ DF (\eta ) \circ \cL_{\im{}}
    \qquad \text{and} \qquad
    \cL_{\im{}} \circ D (\sigma \circ F) (\eta ) = - \sigma \circ \cL_{\im{}} \circ DF (\eta ),
    \]
    so $\sigma \circ F$ satisfies condition (ii) of Theorem~\ref{T:SigmaIndependence} if and only if $DF (\eta ) \circ \cL_{\im{}} = - \cL_{\im{}} \circ DF (\eta )$, since $\sigma$ is invertible.

    (ii) $\Leftrightarrow$ (iv). Let $\zeta\in\sigma(U)$. Since $\sigma$ is linear and involutive,
    \[
    D(F\circ\sigma)(\zeta)=DF(\sigma(\zeta))\circ\sigma.
    \]
    Using $\sigma\circ\cL_{\im{}}=-\cL_{\im{}}\circ\sigma$, we obtain
    \[
    \begin{aligned}
    D(F\circ\sigma)(\zeta)\circ\cL_{\im{}}
    &=-DF(\sigma(\zeta))\circ\cL_{\im{}}\circ\sigma,\\
    \cL_{\im{}}\circ D(F\circ\sigma)(\zeta)
    &=\cL_{\im{}}\circ DF(\sigma(\zeta))\circ\sigma.
    \end{aligned}
    \]
    Since $\sigma$ is bijective, these expressions are equal for every $\zeta\in\sigma(U)$ if and only if condition~(ii) holds throughout $U$.

    The description of $\overline{\hol}_{\im{}} (U)$ follows from (iii) and (iv) applied to $F = \sigma \circ G$ and $F = G \circ \sigma$, using $\sigma \circ \sigma = \mathrm{Id}$, and the independence of $\sigma$ is clear from (ii).
    \end{proof}

    \begin{remark}
    When $n=1$, Proposition~\ref{P:AntiHolo} recovers the classical anti-holomorphic functions: the functions $\overline{G(z)}$ and $G(\overline{z})$ with $G$ holomorphic. Moreover, for any $n \geq 1$, $\hol_{\im{}} (U) \cap \overline{\hol}_{\im{}} (U)$ consists of the locally constant functions: if $DF (\eta )$ both commutes and anticommutes with $\cL_{\im{}}$, then $2\, DF (\eta ) \circ \cL_{\im{}} = 0$, and $DF (\eta ) = 0$ since $\cL_{\im{}}$ is invertible.
    \end{remark}

The next proposition gives the counterpart of Corollary~\ref{C:hol-intersection} for anti-holomorphic functions.
\begin{proposition}\label{P:AntiIntersection}
Let $U\subseteq\mM \mC(n)$ be open, and let
$\tau : \mM \mC(n) \to \mM \mC(n)$ be the involution determined by
\[
\tau(\im{k})=-\im{k},
\qquad 1\leq k\leq n.
\]
Let $F:U\to\mM \mC(n)$ be a function of class $C^1$. Then the following
statements are equivalent.
\begin{enumerate}
    \item[\textnormal{(i)}]
    $F\in\displaystyle\bigcap_{k=1}^{n}
    \overline{\hol}_{\im{k}}(U)$;

    \item[\textnormal{(ii)}]
    $\tau\circ F\in\hol(U)$;

    \item[\textnormal{(iii)}]
    there exists a function $\mu:U\to\mM \mC(n)$ such that
    \[
    DF(\eta)[h]=\mu(\eta)\tau(h),
    \qquad
    \eta\in U,\quad h\in\mM \mC(n).
    \]
\end{enumerate}
Consequently,
\[
\bigcap_{k=1}^{n}\overline{\hol}_{\im{k}}(U)
=
\left\{
\tau\circ G : G\in\hol(U)
\right\}.
\]
Equivalently,
\[
\bigcap_{k=1}^{n}\overline{\hol}_{\im{k}}(U)
=
\left\{
G\circ\tau :
G\in\hol\bigl(\tau(U)\bigr)
\right\}.
\]
\end{proposition}

\begin{proof}
The involution $\tau$ satisfies $\tau(\im{k})=-\im{k}$ for every
$1\leq k\leq n$. Hence $(\tau,\im{k})$ is an $e$-admissible pair for each
$k$. By Proposition~\ref{P:AntiHolo},
\[
F\in\overline{\hol}_{\im{k}}(U)
\quad\Longleftrightarrow\quad
\tau\circ F\in\hol_{\im{k}}(U).
\]
It follows that
\[
F\in\bigcap_{k=1}^{n}\overline{\hol}_{\im{k}}(U)
\quad\Longleftrightarrow\quad
\tau\circ F\in\bigcap_{k=1}^{n}\hol_{\im{k}}(U).
\]
By Corollary~\ref{C:hol-intersection},
\[
\bigcap_{k=1}^{n}\hol_{\im{k}}(U)
=
\hol(U).
\]
This proves the equivalence between \textnormal{(i)} and
\textnormal{(ii)}.

Suppose that \textnormal{(ii)} holds and set $G:=\tau\circ F$. Since
$G\in\hol(U)$, for every $\eta\in U$ there exists
$\lambda(\eta)\in\mM \mC(n)$ such that
\[
DG(\eta)[h]=\lambda(\eta)h,
\qquad
h\in\mM \mC(n).
\]
Since $\tau$ is real-linear, we have
\[
DG(\eta)=\tau\circ DF(\eta).
\]
Applying $\tau$ and using $\tau^2=\operatorname{Id}$, we obtain
\[
DF(\eta)[h]
=
\tau\bigl(\lambda(\eta)h\bigr)
=
\tau\bigl(\lambda(\eta)\bigr)\tau(h).
\]
Thus \textnormal{(iii)} holds with
$\mu(\eta):=\tau\bigl(\lambda(\eta)\bigr)$.

Conversely, suppose that \textnormal{(iii)} holds. Then
\[
\begin{aligned}
D(\tau\circ F)(\eta)[h]
&=
\tau\bigl(DF(\eta)[h]\bigr) =
\tau\bigl(\mu(\eta)\tau(h)\bigr) =
\tau\bigl(\mu(\eta)\bigr)h.
\end{aligned}
\]
Therefore, $D(\tau\circ F)(\eta)$ is $\mM \mC(n)$-linear for every
$\eta\in U$, and hence $\tau\circ F\in\hol(U)$. This proves the
equivalence between \textnormal{(ii)} and \textnormal{(iii)}.

Since $\tau^2=\operatorname{Id}$, condition \textnormal{(ii)} is
equivalent to the existence of $G\in\hol(U)$ such that
$F=\tau\circ G$. This gives the first description of the intersection.
The second description follows from the precomposition characterization
in Proposition~\ref{P:AntiHolo}.
\end{proof}

In view of Proposition~\ref{P:AntiIntersection}, it is natural to define the class
of multicomplex anti-holomorphic functions on $U$ by
\[
\overline{\hol}(U)
:=
\left\{
\tau\circ G : G\in\hol(U)
\right\}.
\]
Proposition~\ref{P:AntiIntersection} then gives the anti-holomorphic counterpart
of Corollary~\ref{C:hol-intersection}:
\[
\overline{\hol}(U)
=
\bigcap_{k=1}^{n}\overline{\hol}_{\im{k}}(U).
\]
In particular, a function is multicomplex anti-holomorphic precisely
when its differential is $\tau$-conjugate $\mM \mC(n)$-linear, in the sense that
\[
DF(\eta)[h]=\mu(\eta)\tau(h)
\]
for some $\mu(\eta)\in\mM \mC(n)$.

We end this section with a result about twisted classes: one may require the differential of $F$ to intertwine two different complex structures. The next proposition shows that this apparent generalization reduces to $\im{}$-holomorphy.

    \begin{proposition}\label{P:Twisted}
    Let $\im{} , \im{}' \in \mI(n)$ with $\ims{} = (\im{}')^2 = -1$, let $\varepsilon \in \{ -1 , 1 \}$, and let $U \subseteq \mM \mC(n)$ be an open set. There exists an $\mI(n)$-preserving automorphism $\psi$ of $\mM \mC(n)$ with $\psi (\im{}) = \varepsilon\, \im{}'$, and for any such $\psi$,
        \begin{align*}
        \{ F \in C^1 (U)\, :\, DF (\eta ) \circ \cL_{\im{}} = \varepsilon\, \cL_{\im{}'} \circ DF (\eta ) , \ \eta \in U \} = \{ \psi \circ G\, :\, G \in \hol_{\im{}} (U) \}.
        \end{align*}
    In particular, every such twisted class consists of real-analytic functions and is obtained from an $\im{}$-holomorphy class by post-composition with a single fixed $\mI(n)$-preserving automorphism.
    \end{proposition}

    \begin{proof}
    For the existence of $\psi$: when $\varepsilon = 1$, an automorphism with $\psi (\im{}) = \im{}'$ was constructed in the proof of Proposition~\ref{P:HoloInclusion}; when $\varepsilon = -1$, compose it with the involution negating one generator $\im{j_0}$ with $j_0$ in the support of $\im{}'$, which maps $\im{}'$ to $-\im{}'$ and preserves $\mI(n)$.

    Since $\psi$ is a $\mathbb R$-algebra automorphism, $\psi \circ \cL_{\im{}} = \cL_{\psi (\im{})} \circ \psi = \varepsilon\, \cL_{\im{}'} \circ \psi$. If $F = \psi \circ G$ with $G \in \hol_{\im{}} (U)$, then for $\eta \in U$, $DF (\eta) = \psi \circ DG (\eta )$ and
    \[
    DF(\eta)\circ\cL_{\im{}}=\psi\circ DG(\eta)\circ\cL_{\im{}}=\psi\circ\cL_{\im{}}\circ DG(\eta) = \varepsilon\, \cL_{\im{}'} \circ \psi \circ DG (\eta ) = \varepsilon\, \cL_{\im{}'} \circ DF (\eta ).
    \]
    Conversely, if $F$ belongs to the twisted class, set $G := \psi^{-1} \circ F$. From $\psi (\im{}) = \varepsilon \im{}'$ we get $\psi^{-1} (\im{}') = \varepsilon \im{}$, hence $\psi^{-1} \circ \cL_{\im{}'} = \varepsilon\, \cL_{\im{}} \circ \psi^{-1}$ and for $\eta \in U$ 
    \[
    DG (\eta ) \circ \cL_{\im{}} = \psi^{-1} \circ DF (\eta ) \circ \cL_{\im{}} = \varepsilon\, \psi^{-1} \circ \cL_{\im{}'} \circ DF (\eta ) = \varepsilon^2\, \cL_{\im{}} \circ DG (\eta ),
    \]
    so $G \in \hol_{\im{}} (U)$. Hence $F = \psi \circ G$ with $G \in \hol_{\im{}} (U)$. 
    \end{proof}

Taking $\im{}' = \im{}$ and $\varepsilon = -1$ in Proposition~\ref{P:Twisted} recovers the anti-holomorphic class. Taking $n = 2$, $\im{} = \im{1}$, $\im{}' = \im{2}$, $\varepsilon = 1$, and $\psi = \omega$ shows that the swap of Section~\ref{SS:Bicomplex} intertwines the two bicomplex classes. 

\subsection{Harmonic functions}

Recall that, alongside the operators $\overline{\partial}_\ell^{\, \sigma , \im{}}$, we introduced their companions
    \begin{equation}\label{Eq:UnbarredWirtinger}
    \partial_\ell^{\, \sigma , \im{}} := \frac{1}{2} \Big( \frac{\partial}{\partial u_\ell} - \im{} \frac{\partial}{\partial v_\ell} \Big) , \qquad 0 \leq \ell \leq m - 1 
    \end{equation}
where $u_{\ell}, v_{\ell}$ are real coordinates associated to a real basis of $B_{\sigma}$, with $(\sigma, \im{})$ an $e$-admissible involution. The next proposition shows that these operators factor a Laplacian, exactly as $\partial_z$ and $\partial_{\overline z}$ factor the planar Laplacian through $\Delta = 4\, \partial_z \partial_{\overline z}$.

    \begin{proposition}\label{P:EllipticFactorization}
    Let $(\sigma , \im{})$ be an $e$-admissible pair, let $U \subseteq \mM \mC(n)$ be open, and let $F : U \ra \mM \mC(n)$ be of class $C^2$. Then, for every $0 \leq \ell \leq m - 1$,
        \begin{equation}\label{Eq:BlockFactorization}
        \partial_\ell^{\, \sigma , \im{}}\, \overline{\partial}_\ell^{\, \sigma , \im{}} = \overline{\partial}_\ell^{\, \sigma , \im{}}\, \partial_\ell^{\, \sigma , \im{}} = \frac{1}{4} \Big( \frac{\partial^2}{\partial u_\ell^2} + \frac{\partial^2}{\partial v_\ell^2} \Big) ,
        \end{equation}
    so that
        \begin{align*}
        \Delta_\sigma := \sum_{\ell = 0}^{m - 1} \Big( \frac{\partial^2}{\partial u_\ell^2} + \frac{\partial^2}{\partial v_\ell^2} \Big) = 4 \sum_{\ell = 0}^{m - 1} \overline{\partial}_\ell^{\, \sigma , \im{}}\, \partial_\ell^{\, \sigma , \im{}} .
        \end{align*}
    Moreover, the real basis $\beta_0 = 1 , \beta_1 , \ldots , \beta_{m - 1}$ of $\cB_\sigma$ may be chosen so that
        \begin{equation}\label{Eq:CanonicalLaplacian}
        \Delta_\sigma = \Delta := \sum_{\vec b \in \mathbb F_2^n} \frac{\partial^2}{\partial x_{\vec b}^2} ,
        \end{equation}
    where $\Delta$ is the standard Laplacian in the canonical coordinates \eqref{Eq:CanoFormMC}; for this choice, $\Delta_\sigma$ is independent of $\sigma$ and of the basis. 
    \end{proposition}

    \begin{proof}
    Since $\im{}$ is a fixed element of $\mM \mC(n)$, it commutes with the constant-coefficient derivations $\partial / \partial u_\ell$ and $\partial / \partial v_\ell$; and since $F$ is of class $C^2$, by Schwarz's theorem, the mixed second partials of each of its real components agree. Expanding the product and using $\ims{} = -1$,
    \[
    \partial_\ell^{\, \sigma , \im{}}\, \overline{\partial}_\ell^{\, \sigma , \im{}} = \frac{1}{4} \Big( \frac{\partial}{\partial u_\ell} - \im{} \frac{\partial}{\partial v_\ell} \Big) \Big( \frac{\partial}{\partial u_\ell} + \im{} \frac{\partial}{\partial v_\ell} \Big) = \frac{1}{4} \Big( \frac{\partial^2}{\partial u_\ell^2} - \ims{} \frac{\partial^2}{\partial v_\ell^2} \Big) = \frac{1}{4} \Big( \frac{\partial^2}{\partial u_\ell^2} + \frac{\partial^2}{\partial v_\ell^2} \Big) .
    \]
    The same computation with the factors reversed gives the identical result. This proves \eqref{Eq:BlockFactorization}, and summation over $\ell$ yields $\tfrac{1}{4} \Delta_\sigma$.

    Endow $\mM \mC(n)$ with the inner product $\langle \cdot , \cdot \rangle$ for which the monomials $\{ \mathbf i^{\vec b} \}_{\vec b \in \mathbb F_2^n}$ form an orthonormal basis. Because $\sigma$ is $\mI(n)$-preserving, it maps each monomial to a signed monomial and permutes the basis up to sign. Hence $\sigma$ is orthogonal, and by Lemma~\ref{L:FixedDecomposition} the subspaces $\cB_\sigma$ and $\im{}\, \cB_\sigma$ are orthogonal. Likewise, since $\im{} \in \mI(n)$, the multiplication $\cL_{\im{}}$ permutes the monomials up to sign by \eqref{Eq:MonomialMult}, so $\cL_{\im{}}$ is an isometry with respect to $\langle \cdot , \cdot \rangle$. 

    Choose an orthonormal basis $\beta_0 = 1 , \beta_1 , \ldots , \beta_{m - 1}$ of $\cB_\sigma$, which is possible with $\beta_0 = 1$ since $1$ is a unit vector of $\cB_\sigma$. Then $\langle \im{} \beta_k , \im{} \beta_\ell \rangle = \langle \beta_k , \beta_\ell \rangle = \delta_{k \ell}$, where $\delta_{kl}$ is the Kronecker symbol, and $\langle \beta_k , \im{} \beta_\ell \rangle = 0$, so the family $\{ \beta_0 , \ldots , \beta_{m - 1} , \im{} \beta_0 , \ldots , \im{} \beta_{m - 1} \}$ is an orthonormal basis of $\mM \mC(n)$. Let $A$ be the matrix of the resulting linear change of coordinates $x = A z$, where $z = ( u_0 , \ldots , u_{m-1} , v_0 , \ldots , v_{m-1} )^\top$; its columns are the coordinates of the $\beta_\ell$ and $\im{} \beta_\ell$ in the monomial basis, so $A^\top A = A A^\top = I$. By the chain rule and $A A^\top = I$,     
    \begin{align*}
    \Delta_\sigma F = \sum_{j = 1}^{2^n} \frac{\partial^2 F}{\partial z_j^2}
    = \sum_{p , q = 1}^{2^n} \Big( \sum_{j = 1}^{2^n} a_{pj}\, a_{qj} \Big) \frac{\partial^2 F}{\partial x_p\, \partial x_q}
    = \sum_{p , q = 1}^{2^n} (A A^\top)_{pq}\, \frac{\partial^2 F}{\partial x_p\, \partial x_q}
    = \sum_{p = 1}^{2^n} \frac{\partial^2 F}{\partial x_p^2} = \Delta F
    \end{align*}
    at every $\eta \in U$, which is \eqref{Eq:CanonicalLaplacian}. 
    \end{proof}

    \begin{corollary}\label{C:HolHarmonic}
    Let $U \subseteq \mM \mC(n)$ be open and let $\im{} \in \mI(n)$ with $\ims{} = -1$. If $F \in \hol_{\im{}} (U)$ or $F \in \overline{\hol}_{\im{}} (U)$, then $\Delta F = 0$, where $\Delta$ is the canonical Laplacian of \eqref{Eq:CanonicalLaplacian}. Equivalently, each of the $2^n$ real components of $F$ is harmonic on $U$.
    \end{corollary}

    \begin{proof}
    By Theorem~\ref{T:SigmaIndependence} and Proposition~\ref{P:AntiHolo}, functions in $\hol_{\im{}} (U)$ and in $\overline{\hol}_{\im{}} (U)$ are real-analytic, hence of class $C^2$. Fix an involution $\sigma$ making $(\sigma , \im{})$ an $e$-admissible pair, and choose the orthonormal basis of Proposition~\ref{P:EllipticFactorization}, so that $\Delta_\sigma = \Delta$. If $F \in \hol_{\im{}} (U)$, then $\overline{\partial}_\ell^{\, \sigma , \im{}} F = 0$ for $0 \leq \ell \leq m - 1$ by \eqref{Eq:SigmaCRCoordinates}, whence, by \eqref{Eq:BlockFactorization},
    \[
    \Delta F = 4 \sum_{\ell = 0}^{m - 1} \partial_\ell^{\, \sigma , \im{}} \big( \overline{\partial}_\ell^{\, \sigma , \im{}} F \big) = 0 .
    \]
    If $F \in \overline{\hol}_{\im{}} (U)$, then $\partial_\ell^{\, \sigma , \im{}} F = 0$ for every $\ell$, and the same identity with the factors reversed gives $\Delta F = 0$. Since $\Delta$ has real coefficients, it acts componentwise on the canonical representation \eqref{Eq:CanoFormMC}, so each real component of $F$ is harmonic.
    \end{proof}

    \begin{remark}
    Proposition~\ref{P:EllipticFactorization} justifies the terminology of Definition~\ref{D:EAdmissible}. The type of the second-order operator produced by the factorization is governed by the sign of $- \ims{}$: for an $e$-admissible pair, $\ims{} = -1$ gives the elliptic blocks $\partial^2 / \partial u_\ell^2 + \partial^2 / \partial v_\ell^2$ and the Laplacian \eqref{Eq:CanonicalLaplacian}. If instead $(\sigma , \jm{})$ is $h$-admissible, with $n \geq 2$, in the sense of the Remarks following Proposition~\ref{P:HoloInclusion}, with $\jm{}^2 = 1$, the same computation gives $- \jm{}^2 = -1$ and the blocks $\partial^2 / \partial u_\ell^2 - \partial^2 / \partial v_\ell^2$, whose sum
    \[
    \sum_{\ell = 0}^{m - 1} \Big( \frac{\partial^2}{\partial u_\ell^2} - \frac{\partial^2}{\partial v_\ell^2} \Big)
    \]
    has signature $(m , m)$. This is an ultrahyperbolic operator when $m \geq 2$; see, for instance, \cite{ CraigWeinstein2008, John1938}.
    \end{remark}

\section{Concluding remarks}\label{S:Conclusion}

We have studied the $\mI(n)$-preserving involutions of the multicomplex algebra
$\mM \mC(n)$ from two complementary angles. On the algebraic side, translating the
problem into $(0,1)$-matrix theory over $\mathbb F_2$ produced the closed-form count
\eqref{Eq:NumberOfInPreservingInvo}, together with the Sign Lemma
(Lemma~\ref{L:SignVectors}). The counting reveals substantially more involutions than
the $2^n$ conjugations traditionally used in multicomplex function theory, and hence
a correspondingly richer supply of Cauchy--Riemann-type systems. On the analytic side,
each elliptic-admissible pair $(\sigma,\im{})$ determines a complex structure
$\cL_{\im{}}$ together with a real form $\cB_\sigma$, and thereby a first-order system.
Theorem~\ref{T:SigmaIndependence} shows that the resulting solution class is governed
by the distinguished unit $\im{}$ alone: the involution $\sigma$ fixes the real
coordinates in which the system is written, but not the function theory it defines.

\bibliographystyle{plain}
\bibliography{Biblio}

@book{Baley,
  title={An Introduction to Multicomplex Spaces and Functions},
  author={Price, G. B.},
  year={1991},
  publisher={M. Dekker: New York, NY, USA},
}

@book{Bicomplex,
   author = {Luna-Elizarraras, M. E. and Shapiro, M. and C. Struppa, D. and Vajiac, A.},
   year = {2015},
   title = {Bicomplex holomorphic functions: the algebra, geometry and analysis of bicomplex numbers},
   publisher = {Springer: Cham},
   series={Frontiers in Mathematics},
   address = {Switzerland},
}

@article{RochonShapiro,
  title={On algebraic properties of bicomplex and hyperbolic numbers},
  author={Rochon, D. and Shapiro, M.},
  journal={Anal. Univ. Oradea, fasc. math},
  volume={11},
  number={71},
  pages={110},
  year={2004}
}

@article{Rochon1,
  title={A generalized {Mandelbrot} set for bicomplex numbers},
  author={Rochon, D.},
  journal={Fractals},
  volume={8},
  number={4},
  pages={355--368},
  year={2000},
  publisher={World Scientific}
}

@article {AlpayDikiVajiac2023,
    AUTHOR = {Alpay, D. and Diki, K. and Vajiac, M.},
     TITLE = {A note on the complex and bicomplex valued neural networks},
   JOURNAL = {Appl. Math. Comput.},
  FJOURNAL = {Applied Mathematics and Computation},
    VOLUME = {445},
      YEAR = {2023},
     PAGES = {Paper No. 127864, 12},
      ISSN = {0096-3003,1873-5649},
   MRCLASS = {30G35 (47A57 68T07)},
  MRNUMBER = {4537587},
}

@article {SangwineTodd2007,
    AUTHOR = {Ell, T. A. and Sangwin,  S. J.},
     TITLE = {Quaternion involutions and anti-involutions},
   JOURNAL = {Comput. Math. Appl.},
  FJOURNAL = {Computers \& Mathematics with Applications. An International
              Journal},
    VOLUME = {53},
      YEAR = {2007},
    NUMBER = {1},
     PAGES = {137--143},
      ISSN = {0898-1221},
   MRCLASS = {11R52 (15A33)},
  MRNUMBER = {2321686},
MRREVIEWER = {Gisele C. Ducati},
}

@article{Segre,
  title={The real representation of complex elements and hyperalgebraic entities ({I}talian)},
  author={Segre, C.},
  journal={Mathematische Annalen},
  volume={40},
  pages={413--467},
  year={1892},
}

@incollection {vajiac2018,
    AUTHOR = {Vajiac, M. B.},
     TITLE = {Norms and moduli on multicomplex spaces},
 BOOKTITLE = {Clifford analysis and related topics},
    SERIES = {Springer Proc. Math. Stat.},
    VOLUME = {260},
     PAGES = {113--140},
 PUBLISHER = {Springer, Cham},
      YEAR = {2018},
   MRCLASS = {30G35},
  MRNUMBER = {3879853},
}

@article{vajiac,
  title={Multicomplex hyperfunctions},
  author={Vajiac, A. and Vajiac, M. B.},
  journal={Complex Variables and Elliptic Equations},
  volume={57},
  number={7--8},
  pages={751--762},
  year={2012},
  publisher={Taylor \& Francis},
}

@article {Lawson2021,
    AUTHOR = {Lawson, J. and Kizil, E.},
     TITLE = {Characterizations of automorphic and anti-automorphic
              involutions of the quaternions},
   JOURNAL = {Linear Multilinear Algebra},
  FJOURNAL = {Linear and Multilinear Algebra},
    VOLUME = {69},
      YEAR = {2021},
    NUMBER = {11},
     PAGES = {1975--1980},
      ISSN = {0308-1087},
   MRCLASS = {16W10 (16H05 16W20)},
  MRNUMBER = {4283114},
MRREVIEWER = {Oliver Villa},
}

@article {RochonTremblay2004,
    AUTHOR = {Rochon, D. and Tremblay, S.},
     TITLE = {Bicomplex quantum mechanics. {I}. {T}he generalized
              {S}chr\"{o}dinger equation},
   JOURNAL = {Adv. Appl. Clifford Algebr.},
  FJOURNAL = {Advances in Applied Clifford Algebras},
    VOLUME = {14},
      YEAR = {2004},
    NUMBER = {2},
     PAGES = {231--248},
      ISSN = {0188-7009},
   MRCLASS = {81Q99 (81R25)},
  MRNUMBER = {2236094},
MRREVIEWER = {Waldyr Alves Rodrigues, Jr.},
}

@article{Parise,
	author = {P.-O. Parisé},
	journal = {arXiv preprint arXiv:2207.06636},
	title = {Involutions of Bicomplex Numbers},
	year = {2022},
 }

@article{BrouilletteRochon2019,
  title={Characterization of the Principal 3\uppercase{D} Slices Related to the Multicomplex \uppercase{M}andelbrot Set},
  author={Brouillette, G. and Rochon, D.},
   JOURNAL = {Adv. Appl. Clifford Algebr.},
  FJOURNAL = {Advances in Applied Clifford Algebras},
  volume={29},
  number={39},
  year={2019},
}

@article {Ceballos2022,
    AUTHOR = {Ceballos, J. and Coloma, N. and Di Teodoro, A. and Ochoa-Tocachi, D. and Ponce, F.},
     TITLE = {Fractional multicomplex polynomials},
   JOURNAL = {Complex Anal. Oper. Theory},
  FJOURNAL = {Complex Analysis and Operator Theory},
    VOLUME = {16},
      YEAR = {2022},
    NUMBER = {4},
     PAGES = {Paper No. 60, 30},
      ISSN = {1661-8254},
   MRCLASS = {26A33 (30G35)},
  MRNUMBER = {4423621},
MRREVIEWER = {Ahmed Mohamed Ahmed El-Sayed},
}

@article {Theaker2017,
    AUTHOR = {Theaker, K. A. and Van Gorder, R. A.},
     TITLE = {Multicomplex wave functions for linear and nonlinear
              {S}chr\"{o}dinger equations},
   JOURNAL = {Adv. Appl. Clifford Algebr.},
  FJOURNAL = {Advances in Applied Clifford Algebras},
    VOLUME = {27},
      YEAR = {2017},
    NUMBER = {2},
     PAGES = {1857--1879},
      ISSN = {0188-7009},
   MRCLASS = {30G35 (35J10)},
  MRNUMBER = {3651546},
MRREVIEWER = {Alessandro Perotti},
}

@incollection {Struppa2012,
    AUTHOR = {Struppa, D. C. and Vajiac, A. and Vajiac, M. B.},
     TITLE = {Holomorphy in multicomplex spaces},
 BOOKTITLE = {Spectral theory, mathematical system theory, evolution
              equations, differential and difference equations},
    SERIES = {Oper. Theory Adv. Appl.},
    VOLUME = {221},
     PAGES = {617--634},
 PUBLISHER = {Birkh\"{a}user/Springer Basel AG, Basel},
      YEAR = {2012},
   MRCLASS = {30G35 (16E05)},
  MRNUMBER = {2954093},
MRREVIEWER = {Dominic Rochon},
}

@article {Pelletier2009,
    AUTHOR = {Garant-Pelletier, V. and Rochon, D.},
     TITLE = {On a generalized {F}atou-{J}ulia theorem in multicomplex
              spaces},
   JOURNAL = {Fractals},
  FJOURNAL = {Fractals. Complex Geometry, Patterns, and Scaling in Nature
              and Society},
    VOLUME = {17},
      YEAR = {2009},
    NUMBER = {3},
     PAGES = {241--255},
      ISSN = {0218-348X},
   MRCLASS = {28A80},
  MRNUMBER = {2567901},
MRREVIEWER = {Byungik Kahng},
}

@article {Luna2017,
    AUTHOR = {Luna-Elizarrar\'{a}s, M. E. and P\'{e}rez-Regalado, C. O. and Shapiro,
              M.},
     TITLE = {On the {L}aurent series for bicomplex holomorphic functions},
   JOURNAL = {Complex Var. Elliptic Equ.},
  FJOURNAL = {Complex Variables and Elliptic Equations. An International
              Journal},
    VOLUME = {62},
      YEAR = {2017},
    NUMBER = {9},
     PAGES = {1266--1286},
      ISSN = {1747-6933},
   MRCLASS = {30G35 (32A10 32A30)},
  MRNUMBER = {3662505},
}

@article{Parise2019,
    AUTHOR = {Brouillette, G. and Paris\'{e}, P.-O. and Rochon,
              D.},
     TITLE = {Tricomplex distance estimation for filled-in {J}ulia sets and
              multibrot sets},
   JOURNAL = {Internat. J. Bifur. Chaos Appl. Sci. Engrg.},
  FJOURNAL = {International Journal of Bifurcation and Chaos in Applied
              Sciences and Engineering},
    VOLUME = {29},
      YEAR = {2019},
    NUMBER = {6},
     PAGES = {1950085, 15},
      ISSN = {0218-1274},
   MRCLASS = {37F99},
  MRNUMBER = {3976460},
}

@article {countinginvo,
    AUTHOR = {Chow, C.-O.},
     TITLE = {Counting involutory, unimodal, and alternating signed
              permutations},
   JOURNAL = {Discrete Math.},
  FJOURNAL = {Discrete Mathematics},
    VOLUME = {306},
      YEAR = {2006},
    NUMBER = {18},
     PAGES = {2222--2228},
      ISSN = {0012-365X},
   MRCLASS = {05A15 (05A05)},
  MRNUMBER = {2255616},
}

@article {countingwr,
    AUTHOR = {Chow, C.-O. and Mansour, T.},
     TITLE = {Counting derangements, involutions and unimodal elements in
              the wreath product {$C_r\wr{\germ S}_n$}},
   JOURNAL = {Israel J. Math.},
  FJOURNAL = {Israel Journal of Mathematics},
    VOLUME = {179},
      YEAR = {2010},
     PAGES = {425--448},
      ISSN = {0021-2172,1565-8511},
   MRCLASS = {05A15 (05E15 20E22)},
  MRNUMBER = {2735050},
MRREVIEWER = {Kent\ E.\ Morrison},
}

@book {Krantz2001,
    AUTHOR = {Krantz, S. G.},
     TITLE = {Function theory of several complex variables},
   EDITION = {Second},
 PUBLISHER = {AMS Chelsea Publishing, Providence, RI},
      YEAR = {2001},
     PAGES = {xvi+564},
      ISBN = {0-8218-2724-3},
   MRCLASS = {32-01},
  MRNUMBER = {1846625},
}

@book{BrackxDelangheSommen1982,
  author    = {Brackx, F. and Delanghe, R. and Sommen, F.},
  title     = {{Clifford Analysis}},
  series    = {Research Notes in Mathematics},
  volume    = {76},
  publisher = {Pitman (Advanced Publishing Program)},
  address   = {Boston, MA},
  year      = {1982},
}

@article{Delanghe1970,
  author  = {Delanghe, R.},
  title   = {On regular analytic functions with values in a {Clifford} algebra},
  journal = {Mathematische Annalen},
  volume  = {185},
  number  = {2},
  pages   = {91--111},
  year    = {1970},
}

@article{Fueter1935,
  author  = {Fueter, R.},
  title   = {{Die Funktionentheorie der Differentialgleichungen $\Delta u = 0$ und $\Delta \Delta u = 0$ mit vier reellen Variablen}},
  journal = {Commentarii Mathematici Helvetici},
  volume  = {7},
  pages   = {307--330},
  year    = {1935},
  doi     = {10.1007/BF01292723},
}

@article{Ryan1982,
  author  = {Ryan, J.},
  title   = {Complexified {Clifford} analysis},
  journal = {Complex Variables, Theory and Application: An International Journal},
  volume  = {1},
  number  = {1},
  pages   = {119--149},
  year    = {1982},
}

@book{Huybrechts2005,
  author    = {Huybrechts, D.},
  title     = {Complex Geometry: An Introduction},
  series    = {Universitext},
  publisher = {Springer-Verlag},
  address   = {Berlin},
  year      = {2005},
  pages     = {xii+309},
  isbn      = {3-540-21290-6},
  doi       = {10.1007/b137952},
  mrnumber  = {2093043},
}

@book{Catoni2008,
  author    = {Catoni, F. and Boccaletti, D. and Cannata, R.
               and Catoni, Vincenzo and Nichelatti, Enrico and Zampetti, Paolo},
  title     = {The Mathematics of {Minkowski} Space-Time: With an
               Introduction to Commutative Hypercomplex Numbers},
  series    = {Frontiers in Mathematics},
  publisher = {Birkh\"{a}user},
  address   = {Basel},
  year      = {2008},
  isbn      = {978-3-7643-8613-9},
  doi       = {10.1007/978-3-7643-8614-6},
}

@misc{PerezRegaladoQuiroga2018,
  author        = {P{\'e}rez-Regalado, C. O. and Quiroga-Barranco, R.},
  title         = {Bicomplex {B}ergman Spaces on Bounded Domains},
  note          = {Preprint},
  year          = {2018},
  eprint        = {1811.00150},
  archivePrefix = {arXiv},
  primaryClass  = {math.FA},
}

@book{Remmert1991,
  author    = {Remmert, R.},
  title     = {Theory of Complex Functions},
  series    = {Graduate Texts in Mathematics},
  volume    = {122},
  publisher = {Springer-Verlag},
  address   = {New York},
  year      = {1991},
  note      = {Translated from the second German edition by Robert B. Burckel},
  doi       = {10.1007/978-1-4612-0939-3},
  isbn      = {978-0-387-97195-7}
}

@article{KoetterKschischang2008,
  author  = {Koetter, R. and Kschischang, F. R.},
  title   = {Coding for errors and erasures in random network coding},
  journal = {IEEE Transactions on Information Theory},
  volume  = {54},
  number  = {8},
  pages   = {3579--3591},
  year    = {2008},
  doi     = {10.1109/TIT.2008.926449},
}

@article{GuptaPandeyRaySamanta2019,
  author  = {Gupta, K. C. and Pandey, S. K. and Ghosh Ray, I.
             and Samanta, S.},
  title   = {Cryptographically significant {MDS} matrices over finite fields:
             A brief survey and some generalized results},
  journal = {Advances in Mathematics of Communications},
  volume  = {13},
  number  = {4},
  pages   = {779--843},
  year    = {2019},
  doi     = {10.3934/amc.2019045},
}

@article{Sudbery1979,
  author  = {Sudbery, A.},
  title   = {Quaternionic analysis},
  journal = {Math. Proc. Cambridge Philos. Soc.},
  volume  = {85},
  number  = {2},
  pages   = {199--224},
  year    = {1979},
  doi     = {10.1017/S0305004100055638},
}

@article{MoisilTeodorescu1931,
  author  = {Moisil, Gr. C. and Teodorescu, N.},
  title   = {Fonctions holomorphes dans l'espace},
  journal = {Mathematica (Cluj)},
  volume  = {5},
  pages   = {142--159},
  year    = {1931},
}

@article{Fueter1936,
  author  = {Fueter, R.},
  title   = {{\"U}ber die analytische Darstellung der regul{\"a}ren Funktionen einer Quaternionenvariablen},
  journal = {Comment. Math. Helv.},
  volume  = {8},
  pages   = {371--378},
  year    = {1936},
  doi     = {10.1007/BF01199562},
}

@book{DelangheSommenSoucek1992,
  author    = {Delanghe, R. and Sommen, F. and Sou\v{c}ek, V.},
  title     = {Clifford Algebra and Spinor-Valued Functions: A Function Theory for the Dirac Operator},
  series    = {Mathematics and its Applications},
  volume    = {53},
  publisher = {Kluwer Academic Publishers Group},
  address   = {Dordrecht},
  year      = {1992},
  doi       = {10.1007/978-94-011-2922-0},
}

@book{GilbertMurray1991,
  author    = {Gilbert, J. E. and Murray, M. A. M.},
  title     = {Clifford Algebras and Dirac Operators in Harmonic Analysis},
  series    = {Cambridge Studies in Advanced Mathematics},
  volume    = {26},
  publisher = {Cambridge University Press},
  address   = {Cambridge},
  year      = {1991},
}

@article{GentiliStruppa2007,
  author  = {Gentili, G. and Struppa, D. C.},
  title   = {A new theory of regular functions of a quaternionic variable},
  journal = {Adv. Math.},
  volume  = {216},
  number  = {1},
  pages   = {279--301},
  year    = {2007},
  doi     = {10.1016/j.aim.2007.05.010},
}

@article{ColomboSabadiniStruppa2009,
  author  = {Colombo, F. and Sabadini, I. and Struppa, D. C.},
  title   = {Slice monogenic functions},
  journal = {Israel J. Math.},
  volume  = {171},
  pages   = {385--403},
  year    = {2009},
  doi     = {10.1007/s11856-009-0055-4},
}

@book{ColomboSabadiniStruppa2011,
  author    = {Colombo, F. and Sabadini, I. and Struppa, D. C.},
  title     = {Noncommutative Functional Calculus: Theory and Applications of Slice Hyperholomorphic Functions},
  series    = {Progress in Mathematics},
  volume    = {289},
  publisher = {Birkh\"{a}user/Springer Basel AG},
  address   = {Basel},
  year      = {2011},
  doi       = {10.1007/978-3-0348-0110-2},
}

@article{CourchesneTremblay2025,
  author  = {Courchesne, D. and Tremblay, S.},
  title   = {Multicomplex Ideals, Modules and {H}ilbert Spaces},
   JOURNAL = {Adv. Appl. Clifford Algebr.},
  FJOURNAL = {Advances in Applied Clifford Algebras},
  year    = {2025},
  volume  = {35},
  number  = {1},
  pages   = {9},
  doi     = {10.1007/s00006-025-01373-y},
  issn    = {1661-4909},
}

@article{John1938,
  author  = {John, F.},
  title   = {The Ultrahyperbolic Differential Equation with Four Independent Variables},
  journal = {Duke Mathematical Journal},
  year    = {1938},
  volume  = {4},
  number  = {2},
  pages   = {300--322},
  doi     = {10.1215/S0012-7094-38-00423-5},
}

@article{CraigWeinstein2008,
  author  = {Craig, W. and Weinstein, S.},
  title   = {On determinism and well-posedness in multiple time dimensions},
  journal = {Proceedings of the Royal Society A: Mathematical, Physical and Engineering Sciences},
  year    = {2009},
  volume  = {465},
  number  = {2110},
  pages   = {3023--3046},
  doi     = {10.1098/rspa.2009.0097},
}

\end{document}